\documentclass[11pt,a4paper]{article}

\usepackage{a4wide}

\usepackage{amsmath}
\usepackage{amssymb}
\usepackage{amsthm}
\usepackage[normalem]{ulem}
\usepackage{color}
\usepackage[left=2.0cm,top=3.0cm,bottom=3.5cm]{geometry}
\usepackage{graphicx}

\usepackage{appendix}

\numberwithin{equation}{section}

\begin{document}

\newtheorem{theorem}{Theorem}[section]
\newtheorem{lemma}[theorem]{Lemma}
\newtheorem{definition}[theorem]{Definition}
\newtheorem{proposition}[theorem]{Proposition}
\newtheorem{corollary}[theorem]{Corollary}
\newtheorem{example}[theorem]{Example}
\newtheorem{remark}[theorem]{Remark}
\newtheorem{conjecture}[theorem]{Conjecture}
\newtheorem*{assumption}{Assumption}

\title{A central limit type theorem
for Gaussian mixture approximations to the nonlinear filtering problem}
\author{Dan Crisan\thanks{Department of Mathematics, Imperial College London, London, SW7 2AZ, UK. Email: d.crisan@imperial.ac.uk},\and Kai Li\thanks{Department of Mathematics, Uppsala University, Box 480, Uppsala, 75106, Sweden. Email: kai.li@math.uu.se}}

\maketitle

\begin{abstract} 
Approximating the solution of the nonlinear filtering problem with Gaussian mixtures has been a very popular method since the 1970s. 
However, the vast majority of such approximations are introduced in an ad-hoc manner without theoretical grounding. This work is a continuation of \cite{Crisan and Li, Crisan and Li 2}, where we described a rigorous Gaussian mixture approximation to the solution of the filtering problem. We deduce here a refined estimate of the rate of convergence of the approximation. We do this by proving a central limit type theorem for the error process. We also find the optimal variances of the Gaussian measures are of order $1/\sqrt n$. This implies, in particular, that the mean square error of the approximation as defined in \cite{Crisan and Li, Crisan and Li 2} is of order $1/n$.
\end{abstract}

\section{Introduction}

The stochastic filtering problem deals with the estimation of an evolving dynamical system, called the \emph{signal}, based on \emph{partial observations} and a priori stochastic model. The signal is modelled by a stochastic process denoted by $X=\{X_t,\ t\geq0\}$, defined on a probability space $(\Omega,\mathcal{F},\mathbb{P})$. The signal process is not available to observe directly; instead, a partial observation is obtained and it is modelled by a process $Y=\{Y_t,\ t\geq0\}$.
The information available from the observation up to time $t$ is defined as the filtration $\mathcal{Y}=\{\mathcal{Y}_t,\ t\geq0\}$ generated by the observation process $Y$.
In this setting, we want to compute $\pi_t$ --- the conditional distribution of $X_t$ given $\mathcal{Y}_t$.

The description of a numerical approximation for $\pi_t$ should contain the following three parts: the class of approximations; the law of evolution of the approximation; and the method of measuring the approximating error. Gaussian mixtures approximations are numerical schemes that approximate  $\pi_t$ with random measures of the form
$$
\sum_ja_j(t)\Gamma_{v_j(t),\omega_j(t)},
$$
where $a_j(t)$ is the weight of the Gaussian (generalised) particle, $\Gamma_{v_j(t),\omega_j(t)}$ is the Gaussian measure with mean $v_j(t)$ and covariance matrix $\omega_j(t)$. The evolution of the weights, the mean and the covariance matrices satisfy certain stochastic differential equations which are numerically solvable. 

Studies of Gaussian mixtures approximations in the context of Bayesian estimation have been developing for nearly fifty years since 1970s (see, for example, \cite{Crisan and Li 2} for a survey of the existing work). However, not until recently can we see a theoretical analysis and $L^2$-convergence rate for such approximating system obtained by Crisan and Li (\cite{Crisan
and Li, Crisan and Li 2}). In addition to the $L^2$-convergence, it is also of great importance that one can recalibrate the error of the approximation and characterise its exact convergence rate, in other words, prove a central limit theorem type result of such approximation. 

Various other approximations to the nonlinear filtering problems have been shown to satisfy central limit type theorems. Del Moral, Guionnet, and Miclo (see \cite{Del Moral CLT}, \cite{Del Moral and Guionnet}, \cite{Del Moral and Miclo}) deduced central limit type results (CLT) for unweighted particle filters using the interacting particle systems. Crisan and Xiong (\cite{Crisan and Xiong}) proved a CLT result for the classical nonlinear filtering case and obtained the rate as $n^{(1-\alpha)/2}$ for any $\alpha>0$; and this result was later improved by Xiong and Zeng (\cite{Xiong and Zeng}) up to $n^{1/2}$. Similar CLT results were also obtained for the discrete time filtering framework by Chopin (\cite{Chopin}) and Kunsch (\cite{Kunsch}).

However, to the authors' knowledge, there has been no theoretical analysis of the convergence in distribution for the Gaussian mixture approximations to the filtering problem, and no corresponding central limit type result was proven for this type of approximations. The main purpose of this paper is to fill this gap and obtain a CLT result for such approximation.

\subsection{Contribution of the paper}
This paper is a continuation of the work done in \cite{Crisan and Li 2}. In particular, we deduce here a central limit theorem for the algorithm presented in \cite{Crisan and Li, Crisan and Li 2}. To be specific, let $\pi=\{\pi_t;t\geq0\}$ be the conditional distribution and $\pi^{n,\varepsilon}=\{\pi_t^{n,\varepsilon};t\geq0\}$ be the approximation of the conditional distribution constructed in \cite{Crisan and Li 2} (and in Section 3 in this paper) using mixtures of Gaussian measures, where $n$ is the number of Gaussian measures and $\varepsilon$ is a positive parameter measuring the amount of ``Gaussianity'' (see discussion after \eqref{eq.appro_with_variance} for details). We obtain a central limit type result and show that the recalibrated error converges in distribution to a unique measure-valued process as $n$ increases; in addition, we find the optimal value for $\varepsilon$.

To do this we introduce the following measure-valued processes $\bar U^{n,\varepsilon}=\{\bar U_t^{n,\varepsilon};t\geq0\}$ and $U^{n,\varepsilon}=\{U_t^{n,\varepsilon};t\geq0\}$ as
\begin{align}
\bar U_t^{n,\varepsilon}=n^\varepsilon(\pi_t^{n,\varepsilon}-\pi_t)\quad \text{and}\quad U_t^{n,\varepsilon}=n^\varepsilon(\rho_t^{n,\varepsilon}-\rho_t)\nonumber,
\end{align}
where $\rho$ ($\rho^{n,\varepsilon}$) is the unnormalised version of $\pi$ ($\pi^{n,\varepsilon}$) (see Section 3 for details). Then we have the following.
\begin{theorem}\label{thm.main_CLT}
The $L^2$-convergence rate of $\pi^{n,\varepsilon}\ (\rho^{n,\varepsilon})$ to $\pi\ (\rho)$ is ${\left(1\over n\right)}^{\min\{2\varepsilon,1\}}$ for $\varepsilon>0$. When $0<\varepsilon\leq1/2$, for each $\varepsilon$, there is a unique measure-valued process $U^\varepsilon=\{U_t^\varepsilon;t\geq0\}$ solving the following stochastic PDE, given any test function in $\varphi\in C_b^6({\mathbb R}^d)$:
$$
U_t^\varepsilon(\varphi)=U_0^\varepsilon(\varphi)+\int_0^tU_s^\varepsilon(A\varphi)ds+\int_0^tU_s^\varepsilon(h\varphi)dY_s+\Lambda_t^{\varphi,\varepsilon},
$$ 
where the definitions of the operator $A$, the function $h$ and $\Lambda^\varphi$ can be found in subsequent sections; and $U^{n,\varepsilon}$ forms a tight sequence and converges in distribution to the process $U^\varepsilon$. In addition, $\bar U^{n,\varepsilon}$ converges in distribution to a measure-valued process $\bar U^\varepsilon=\{\bar U_t^\varepsilon;t\geq0\}$, which is defined by
$$
\bar U_t^{\varepsilon}(\varphi)=\frac{1}{\rho_t(\mathbf 1)}\left(U_t^\varepsilon(\varphi)-\pi_t(\varphi)U_t^\varepsilon(\mathbf 1)\right).
$$
When $\varepsilon>1/2$, the process $\{U^{n,\varepsilon}\}_n\ (\{\bar U^{n,\varepsilon}\}_n)$ is divergent. In other words, the central limit theorem is obtained when $\varepsilon\in(0,1/2]$, and among this range $\varepsilon=1/2$ gives the optimal $L^2$-convergence rate.
\end{theorem}

\begin{remark}
The proof of the $L^2$-convergence rate of $\pi^{n,\varepsilon}\ (\rho^{n,\varepsilon})$ to $\pi\ (\rho)$ can be found in Section 4 of \cite{Crisan and Li 2}, hence we will not prove this part of Theorem \ref{thm.main_CLT} in this paper.
\end{remark}

The following is a summary of the contents of the paper.

In Section 2, we review the key results of stochastic filtering theory. The filtering framework is introduced first, with the focus on the problems where the signal $X$ and observation $Y$ are diffusion processes and the filtering equations are presented.

Section 3 contains the description of the generalised particle filters with Gaussian mixtures. These approximations use mixtures of Gaussian measures which will be set out, with the aim of estimating the solutions to the Zakai and the Kushner-Stratonovich equations. The Multinomial branching algorithm is chosen to be the associated correction mechanism.

Sections 4 and 5 contain the main result of the paper, which is the central limit theorem associated to the approximating system. The analysis is proceeded in a standard manner. In Section 4, based on the evolution equations of the approximating systems derived in \cite{Crisan and Li 2}, the error between the Gaussian mixture approximation and the true solution is recalibrated and shown to be a tight sequence. In section 5, we find its limit in distribution and show this limiting process is unique. 

This paper is concluded in Section 6 and with an Appendix which contains some additional results required in the main body of the paper.

\subsection{Notations}
$\bullet$
$\mathbb R^d$ - the $d$-dimensional Euclidean space.\newline
$\bullet$
$\overline{\mathbb R^d}$ - the one-point compactification of $\mathbb R^d$.\newline
$\bullet$
$\left(\mathbb R^d,\mathcal B(\mathbb R^d)\right)$ - the state space of the signal. $\mathcal
B(\mathbb R^d)$ is the associated Borel $\sigma$-algebra.\newline
$\bullet$
$B(\mathbb R^{d})$ - the space of bounded $\mathcal B(\mathbb R^d)$-measurable
functions from $\mathbb R^d$ to $\mathbb R$.\newline
$\bullet$
$\mathcal P\left(\mathbb R^d\right)$ - the family of Borel probability measures
on space $\mathbb R^d$.\newline
$\bullet$
$C_b(\mathbb R^d)$ - the space of bounded continuous functions on $\mathbb
R^d$.\newline
$\bullet$
$C_b^m(\mathbb R^d)$ - the space of bounded continuous functions on $\mathbb
R^d$ with bounded derivatives to order $m$.\newline
$\bullet$
$C_0^m(\mathbb R^d)$ - the space of continuous functions on $\mathbb R^d$,
vanishing at infinity with continuous partial derivatives up to order $m$.\newline
$\bullet$
$\Vert\cdot\Vert$ - the Euclidean norm for a $d\times p$ matrix $a$, $\Vert a\Vert=\sqrt{\sum_{i=1}^d\sum_{j=1}^pa_{ij}^2}$.\newline
$\bullet$
$\|\cdot\|_\infty$ - the supremum norm for $\varphi:\mathbb R^d\rightarrow\mathbb
R$: $\|\varphi\|_\infty=\sup_{x\in\mathbb R^d}\|\varphi(x)\|$.\newline
$\bullet$
$\|\cdot\|_{m,\infty}$ - the norm such that for  $\varphi$ on $\mathbb R^d$,
$\|\varphi\|_{m,\infty}=\sum_{|\alpha|\leq m}\sup_{x\in\mathbb R^d}|D_\alpha\varphi(x)|$,
where $\alpha=(\alpha^1,\ldots,\alpha^d)$ is a multi-index and $D_\alpha=(\partial_1)^{\alpha_1}\cdots(\partial_d)^{\alpha_d}$.\newline
$\bullet$
$\mathcal M_{F}(\mathbb R^d)$ - the set of finite measures on $\mathbb R^d$.\newline
$\bullet$
$\mathcal M_{F}(\overline{\mathbb R^d})$ - the set of finite measures on
$\overline{\mathbb R^d}$.\newline
$\bullet$
$D_{\mathcal M_{F}(\mathbb R^d)}[0,T]$ - the space of c\`adl\`ag functions
(or right continuous functions with left limits) $f:[0,T]\rightarrow\mathcal
M_{F}(\mathbb R^d)$.\newline
$\bullet$
$D_{\mathcal M_{F}(\mathbb R^d)}[0,\infty)$ - the space of c\`adl\`ag functions
(or right continuous functions with left limits) $f:[0,\infty)\rightarrow\mathcal
M_{F}(\mathbb R^d)$.

\section{The Filtering Problem and Key Result}
Let $(\Omega,\mathcal{F},\mathbb{P})$ be a probability space together with a filtration $(\mathcal{F}_t)_{t\geq0}$ which satisfies the usual conditions. On $(\Omega,\mathcal{F},\mathbb{P})$ we consider an $\mathcal{F}_t$-adapted process $X=\{X_t; t\geq0\}$  taking values on $\mathbb{R}^d$.
To be specific,
let $X=(X^i)_{i=1}^d$ be the solution of a $d$-dimensional stochastic differential
equation driven by a $p$-dimensional Brownian motion $V=(V^j)_{j=1}^p$:
\begin{equation}\label{eq.signal_is_diffusion}
X_t^i=X_0^i+\int_0^tf^i(X_s)ds+\sum_{j=1}^p\int_0^t\sigma^{ij}(X_s)dV_s^j,\quad\quad
i=1,\ldots,d.
\end{equation}
We assume that both $f=(f^i)_{i=1}^d:\mathbb{R}^d\rightarrow\mathbb{R}^d$
and $\sigma=(\sigma^{ij})_{i=1,\ldots,d;j=1,\ldots,p}:\mathbb{R}^d\rightarrow\mathbb{R}^{d\times
p}$ are globally Lipschitz. Under the globally Lipschitz condition, \eqref{eq.signal_is_diffusion}
has a unique solution (e.g., Theorem 5.2.9 in \cite{Karatzas and Shreve}). 

Let $h=(h_i)_{i=1}^m:\mathbb{R}^d\rightarrow\mathbb{R}^m$ be a bounded measurable function. Let $W$ be a standard $\mathcal{F}_t$-adapted $m$-dimensional Brownian motion on  $(\Omega,\mathcal{F},\mathbb{P})$ independent of $X$, and $Y$ be the process which satisfies the following evolution equation
\begin{equation}\label{eq.observation}
Y_t=Y_0+\int_0^th(X_s)ds+W_t,
\end{equation}
This process $Y=\{Y_t;t\geq0\}$ is called the \textit{observation} process. Let $\{\mathcal{Y}_t,t\geq0\}$ be the usual augmentation of the filtration associated with the process $Y$, viz
$\mathcal{Y}_t=\sigma(Y_s,s\in[0,t])\vee\mathcal{N}.$

As stated in the introduction, the filtering problem consists in determining the conditional distribution $\pi_t$ of the signal $X$ at time $t$ given the information accumulated from observing $Y$ in the interval $[0,t]$; that is, for $\varphi\in B(\mathbb{R}^d)$,
\begin{equation}
\pi_t(\varphi)\triangleq\int_{\mathbb{R}^d}\varphi(x)\pi_t(dx)=\mathbb{E}[\varphi(X_t)\mid\mathcal{Y}_t].
\end{equation}

Throughout this paper we make the following assumption.
\begin{assumption}[A]
Assume that the coefficients $f^i$ and $\sigma^{ij}$ are bounded and six times differentiable, and $h^i$ is twice differentiable and has bounded derivatives. That is, $f^i, \sigma^{ij}\in C_b^6(\mathbb R^d)$ and $h^i\in C_b^2(\mathbb R^d)$. 
\end{assumption}
Let $\tilde{\mathbb P}$ be a new probability measure on $\Omega$, under which the process $Y$ is a Brownian motion. To be specific, let $Z=\{Z_t,t\geq0\}$ be the process defined by
\begin{equation}\label{eq.definition_Z}
Z_t=\exp\left(-\sum_{i=1}^m\int_0^t h^i(X_s)dW_s^i-\frac{1}{2}\sum_{i=1}^m\int_0^t h^i(X_s)^2ds\right),\quad t\geq0;
\end{equation}
and we introduce a probability measure $\tilde{\mathbb{P}}^t$ on $\mathcal{F}_t$ by specifying its Radon-Nikodym derivative with respect to $\mathbb{P}$ to be given by $Z_t$.
We finally define a probability measure $\tilde{\mathbb{P}}$ which is equivalent to $\mathbb{P}$ on $\bigcup_{0\leq t<\infty}\mathcal{F}_t$.
Then we have the following Kallianpur-Striebel formula (see \cite{Kallianpur and Karandikar})
\begin{equation}
\pi_t(\varphi)=\frac{\rho_t(\varphi)}{\rho_t(\mathbf 1)}\quad\quad \tilde{\mathbb P}(\mathbb{P})-a.s.\quad\text{for}\ \varphi\in B(\mathbb{R}^d),
\end{equation}
where $\rho_t$ is an $\mathcal Y_t$-adapted measure-valued process satisfying the following \textbf{Zakai Equation} (see \cite{Zakai}).
\begin{equation}\label{eq.zakai_equation}
\rho_t(\varphi)=\pi_0(\varphi)+\int_0^t\rho_s(A\varphi)ds+\int_0^t\rho_s(\varphi h^\top)dY_s,\quad\tilde{\mathbb{P}}-a.s.\quad\forall t\geq0
\end{equation}
for any $\varphi\in\mathcal{D}(A)$. 
In \eqref{eq.zakai_equation}, operator $A$ is the infinitesimal generator associated with the signal process $X$
\begin{equation}\label{eq.generator_A}
A=\sum_{i=1}^d f^i\frac{\partial}{\partial x^i}+\sum_{i=1}^d\sum_{j=1}^da^{ij}\frac{\partial^2}{\partial
x^i\partial x^j},
\end{equation}   
where $a=(a^{ij})_{i,j=1,\ldots,d}:\mathbb{R}^d\rightarrow\mathbb{R}^{d\times
d}$ is the matrix-valued function defined as $a=\frac{1}{2}\sigma^\top\sigma$; and $\mathcal{D}(A)$ is the domain of $A$. 

Also the process $\rho=\{\rho_t;t\geq0\}$ is called the unnormalised conditional distribution of the signal.

In the following we will obtain the central limit theorem for the associated generalised particle filters with Gaussian mixtures. We denote by $\pi^{n,\varepsilon}=\{\pi_t^{n,\varepsilon};t\geq0\}$ the approximating measures of the solution of the filtering problem, where $n$ is the number of Gaussian measures in the approximating system, and $\varepsilon$ is a parameter measuring the amount of ``Gaussianity'' of the generalised particles.

\section{Gaussian Mixtures Approximation}\label{gaussian}

For ease of notations, we assume, hereinafter from this section, that the state space of the signal is one-dimensional. For clarity we describe the Gaussian mixture approximation introduced in \cite{Crisan and Li, Crisan and Li 2} below in this section. All the results presented here can be extended without significant technical difficulties to the multi-dimensional case. 

Firstly, we let $\Delta=\{0=\delta_0<\delta_1<\cdots<\delta_N=T\}$ be an equidistant partition of the interval $[0,T]$ with fixed equal length,
with $\delta_i=i\delta,\ i=1,\ldots,N$; and $N=\frac{T}{\delta}$. We also denote $n$ by the number of generalised particles in the system. The approximating algorithm is then introduced as follows.

\textbf{Initialisation}: At time $t=0$, the particle system consists of $n$ Gaussian measures  all with equal weights $1/n$, initial means $v_j^n(0)$, and initial variances $\omega_j^n(0)$, for $j=1,\ldots,n$; denoted by 
$\Gamma_{v_j^n(0),\omega_j^n(0)}$. 
The approximation of $\pi_0^{n,\varepsilon}$ has the form
\begin{equation}
\pi_0^{n,\varepsilon}\triangleq{\frac{1}{n}}\sum_{j=1}^n\Gamma_{v_j^n(0),\omega_j^n(0)}.
\end{equation}
We will, for $1\leq j\leq n$, choose the initial variances $\omega_j^n(0)=\alpha\beta$ and be given the initial means $v_j^n(0)$, where $\varepsilon$, $\alpha$ and $\beta$ are some parameters defined later in this section.

\textbf{Recursion}: During the interval $t\in[i\delta,(i+1)\delta)$, $i=1,\ldots,N,$ the approximation $\pi^{n,\varepsilon}$ of the normalised conditional distribution $\pi$ will take the form
\begin{equation}\label{eq.gaussian_mixture_approximation}
\pi_t^{n,\varepsilon}\triangleq\sum_{j=1}^n\bar a_j^n(t)\Gamma_{v_j^n(t),\omega_j^n(t)},
\end{equation}
where $v_j^n(t)$ denotes the mean and $\omega_j^n(t)$ denotes the variance of the Gaussian measure $\Gamma_{v_j^n(t),\omega_j^n(t)}$, and $a_j^n(t)$ is the (unnormalised) weight of the particle, and $$\bar a_j^n(t)=\frac{a_j^n(t)}{\sum_{k=1}^na_k^n(t)}$$
is the normalised weight. Obviously, each particle is characterised by the triple process $(a_j^n,v_j^n,\omega_j^n)$ which is chosen to evolve as
\begin{equation}\label{eq.appro_with_variance}
\left\{
\begin{array}{lll}
a_j^n(t)=1+\int_{i\delta}^t a_j^n(s)h(v_j^n(s))dY_s,\\
v_j^n(t)=v_j^n(i\delta)+\int_{i\delta}^tf\left(v_j^n(s)\right)ds
+\sqrt{1-\alpha}\int_{i\delta}^t\sigma\left(v_j^n(s)\right)dV_s^{(j)},\\
 \omega_j^n(t)=\alpha\left(\beta+\int_{i\delta}^t\sigma^2\left(v_j^n(s)\right)ds\right),
\end{array}
\right.
\end{equation}
where $\{{V^{(j)}}\}_{j=1}^n$ are mutually independent Brownian motions and  independent of $Y$. The parameter $\alpha$ is a real number in the interval $[0,1]$. Here we choose $\alpha=n^{-\varepsilon}$, where $\varepsilon\in[0,\infty]$ is
a non-negative parameter measuring the ``Gaussianity'' of the generalised particles. To be specific, the variance of each Gaussian (generalised) particle can be controlled by the value of $\varepsilon$. For $\varepsilon=\infty$ ($\alpha=0$) we recover the classic particle approximation (see, for example, Chapter 9 in \cite{Bain and Crisan}) with the Gaussian measures degenerated to Dirac measures; for $\varepsilon=0$ ($\alpha=1$) we have the largest possible variances and the means of the Gaussian measures evolve deterministically (the stochastic term is eliminated). Therefore we can normally restrict ourselves to the cases where $\varepsilon\in(0,\infty)$. One of the purposes of this paper is to find the optimal value for $\varepsilon$. The parameter $\beta$ is a positive real number, which we call the \emph{smoothing parameter}, ensures that the approximating measure has smooth density at the branching/correction times.

\textbf{Correction}: At the end of the interval $[i\delta,(i+1)\delta)$, immediately prior to the correction step, each Gaussian measure is replaced by a random number of offsprings, which are Gaussian measures with mean $X_j^n((i+1)\delta)$ and variance $\alpha\beta$, where the mean $X_j^n$ is a normally distributed random variable, i.e.
$$
X_j^n((i+1)\delta)\sim\mathcal N\left(v_j^n(i+1)\delta_-,\omega_j^n(i+1)\delta_-\right),\quad j=1,\ldots,n;
$$ 
where by $(i+1)\delta_-$ we denote the time immediately prior to correction. 
We denote by $o_j^{n,(i+1)\delta}$ the number of ``offsprings'' produced by the $j$th generalised particle. The total number of offsprings is fixed to be $n$ at each correcting event.

After correction all the particles are re-indexed from 1 to $n$ and all of
the unnormalised weights are re-initialised back to 1; and the particles
evolve following \eqref{eq.appro_with_variance} again. 
The recursion is repeated $N$ times until we reach the terminal time $T$,
where we obtain the approximation $\pi_T^n$ of $\pi_T$.

We refer to \cite{Crisan and Li 2} for a brief explanation why we should introduce correction mechanism. In the following we adopt the correction algorithm called the Multinomial Resampling to determine the number of offsprings $\{o_j^n\}_{j=1}^n$ (see, for example, \cite{Obanubi}). 
The multinomial resampling algorithm essentially consists of sampling $n$ times with replacement at correction times. At branching times, we sample $n$ times (with replacement) from the population of Gaussian random variables $X_j^n((i+1)\delta)$ (with means $v_j^n((i+1)\delta_-)$ and variances $\omega_j^n((i+1)\delta_-)),j=1,\ldots,n$ according to the multinomial probability distribution given by the corresponding normalised weights $\bar a_j^n((i+1)\delta_-),j=1,\ldots,n$. Therefore, by definition of multinomial distribution, $o_j^{n,(i+1)\delta}$ is the number of times $X_j^n((i+1)\delta)$ is chosen at time $(i+1)\delta$; that is to say, $o_j^{n,(i+1)\delta}$ is the number of offspring produced by this Gaussian random variable.

We then define the process $\xi^n=\{\xi_t^n;t\geq0\}$ by
$$
\xi_t^n\triangleq\left(\prod_{i=1}^{[t/\delta]}\frac{1}{n}\sum_{j=1}^na_j^{n,i\delta}\right)\left(\frac{1}{n}\sum_{j=1}^na_j^n(t)\right).
$$ 
Then $\xi^n$ is a martingale and by Exercise 9.10 in \cite{Bain and Crisan}
we know for any $t\geq0$ and $p\geq1$, there exist two constants $c_1^{t,p}$
and $c_2^{t,p}$ which depend only on $t$, $p$, and $\max_{k=1,\ldots,m}\|h_k\|_{0,\infty}$,
such that
\begin{equation}
\sup_{n\geq0}\sup_{s\in[0,t]}\tilde{\mathbb{E}}\left[(\xi_s^n)^p\right]\leq
c_1^{t,p},\qquad\text{}
\end{equation} 
and
\begin{equation}\label{eq.bound_for_xi_2}
\max_{j=1,\ldots,n}\sup_{n\geq0}\sup_{s\in[0,t]}\tilde{\mathbb{E}}\left[(\xi_s^na_j^n(s))^p\right]\leq
c_2^{t,p}.
\end{equation}
We use the martingale $\xi^n$ to linearise $\pi^{n,\varepsilon}$, to be specific, we define
the measure-valued process $\rho^{n,\varepsilon}=\{\rho_t^{n,\varepsilon}:t\geq0\}$ to be
\begin{align}
\rho_t^{n,\varepsilon}\triangleq\xi_t^n\pi_t^{n,\varepsilon}=\frac{\xi_{[t/\delta]\delta}^n}{n}\sum_{j=1}^na_j^n(t)\Gamma_{v_j^n(t),\omega_j^n(t)}.
\end{align}

Define $U=\{U_t^{n,\varepsilon}:t\geq0\}$ to be the measure-valued process
\begin{equation}\label{eq.rescaled_error}
U_t^{n,\varepsilon}\triangleq n^\varepsilon(\rho_t^{n,\varepsilon}-\rho_t),
\end{equation}
and we aim to find an appropriate range for $\varepsilon$ and show that, with the right choice of $\varepsilon$, the corresponding $\{U^{n,\varepsilon}\}_n$ converges in distribution to a process
$U^\varepsilon$, which is uniquely identified as the solution of a certain martingale
problem. This implies that for any continuous and bounded test function,
\begin{equation}
\lim_{n\rightarrow\infty} n^{\varepsilon}(\rho_t^{n,\varepsilon}(\varphi)-\rho_t(\varphi))=U_t^\varepsilon(\varphi);
\end{equation}
hence the error of the approximations $\rho_t^{n,\varepsilon}(\varphi)$ of $\rho_t(\varphi)$
is roughly ${U_t^\varepsilon(\varphi)}n^{-\varepsilon}$.

By Proposition 4.1 in \cite{Crisan and Li 2}, we have
\begin{align}\label{eq.evolution_for_U^n}
U_t^{n,\varepsilon}(\varphi)
=&U_0^{n,\varepsilon}(\varphi)+\int_0^tU_s^{n,\varepsilon}(A\varphi)ds+\int_0^tU_t^{n,\varepsilon}(h\varphi)dY_s+n^\varepsilon M_{[t/\delta]}^{n,\varphi}+n^\varepsilon B_t^{n,\varphi},
\end{align}
in \eqref{eq.evolution_for_U^n},
\begin{align}\label{eq.M_n_varepsilon}
n^\varepsilon M_{[t/\delta]}^{n,\varphi}&=n^{\varepsilon-1}\sum_{i=0}^{[t/\delta]}\xi_{i\delta}^n\sum_{j=1}^n\Bigg[o_j^{n,i\delta}\int_{\mathbb R}\varphi(x)\frac{e^{-\frac{(x-X_j^n(i\delta))^2}{2\alpha\beta}}}{\sqrt{2\pi\alpha\beta}}dx-n\bar a_j^n(i\delta-)\int_{\mathbb R}\varphi(x)\frac{e^{-\frac{(x-v_j^n(i\delta-))^2}{2\omega_j^n(i\delta-)}}}{\sqrt{2\pi\omega_j^n(i\delta-)}}dx\Bigg],\\
n^\varepsilon B_t^{n,\varphi}&=n^{\varepsilon-1}\sum_{j=1}^n\int_{0}^{t}\xi_{[s/\delta]\delta}^na_j^n(s)\Big[R_{s,j}^1(\varphi)ds+R_{s,j}^2(\varphi)dY_s+R_{s,j}^3(\varphi)dV_s^{(j)}\Big];
\end{align}
where
\begin{align}
R_{s,j}^1(\varphi)=&\omega_j^n(s)\left[\frac{1}{2}(f\varphi''')(v_j^n(s))+\frac{\alpha}{4}(\sigma\varphi^{(4)})(v_j^n(s))+2\alpha\sigma^2(v_j^n(s))I_{4,j}^{(4)}(\varphi)-I_j(A\varphi)\right]\nonumber\\
+&(\omega_j^n(s))^2\left[f(v_j^n(s))I_{4,j}^{(5)}(\varphi)+\frac{\alpha\sigma^2(v_j^n(s))}{2\sqrt{\omega_j^n(s)}}I_{5,j}(\varphi)+\frac{1-\alpha}{2}\sigma^2(v_j^n(s))I_{4,j}^{(6)}(\varphi)\right],\label{eq.R^1}
\\
R_{s,j}^2(\varphi)=&\omega_j^n(s)\left[\frac{1}{2}h(v_j^n(s))\varphi''(v_j^n(s))-I_j(h\varphi)\right]
+(\omega_j^n(s))^2h(v_j^n(s))I_{4,j}^{(4)}(\varphi),\label{eq.R^2} \\
R_{s,j}^3(\varphi)=&\sqrt{1-\alpha}\Bigg[\sigma(v_j^n(s))\varphi'(v_j^n(s))+\frac{1}{2}\omega_j^n(s)\sigma(v_j^n(s))\varphi'''(v_j^n(s))\nonumber\\
&\qquad\quad+(\omega_j^n(s))^2\sigma(v_j^n(s))I_{4,j}^{(5)}(\varphi)\Bigg];
\label{eq.R^3}
\end{align}
and
\begin{align}
I_{4,j}^{(k)}(\varphi)=&\int_{\mathbb{R}}\frac{y^4e^{\frac{-y^2}{2}}}{\sqrt{2\pi}}\int_0^1\varphi^{(k)}\left(v_j^n(s)+uy\sqrt{\omega_j^n(s)}\right)\frac{(1-u)^3}{6}dudy,\quad\text{for}\
k=4,5,6;\nonumber\\
I_{5,j}(\varphi)=&\int_{\mathbb{R}}\frac{y^5e^{\frac{-y^2}{2}}}{\sqrt{2\pi}}\int_0^1\varphi^{(5)}\left(v_j^n(s)+uy\sqrt{\omega_j^n(s)}\right)\frac{u(1-u)^3}{6}dudy;\nonumber\\
I_j(\psi)=&\int_{\mathbb{R}}\frac{y^2e^{\frac{-y^2}{2}}}{\sqrt{2\pi}}\int_0^1(\psi)''\left(v_j^n(s)+uy\sqrt{\omega_j^n(s)}\right)(1-u)dudy,\quad\text{for}\
\psi=A\varphi,h\varphi.\nonumber
\end{align}

The machinery used to prove the convergence in distribution for $U^{n,\varepsilon}$ consists of two steps. In step one we show the tightness property of $U^{n,\varepsilon}$. In step two we show that any convergent subsequence of $U^{n,\varepsilon}$ has a limit $U^\varepsilon$ (in distribution) that is the unique solution of a certain martingale problem. These two steps are done in the following two sections.

\subsection*{Discussion on the parameter $\varepsilon$}
Before proceeding to the proof of convergence in distribution, here we discuss the influence of $\varepsilon$ on the convergence of the approximating algorithm. From Section 4 in \cite{Crisan and Li 2} it can be concluded that the $L^2$-convergence rate of the Gaussian mixture approximation is $\left(\frac{1}{n}\right)^{\min\{2\varepsilon,1\}}$. It means that for $\varepsilon\in(0,1/2]$ the convergence rate becomes better as $\varepsilon$ increases, and it then stays at $n^{-1}$ for any $\varepsilon>1/2$.

Following the proof of Lemma 4.7 in \cite{Crisan and Li 2}, it can be shown that, when $\varepsilon>1/2$, $n^\varepsilon M_{[t/\delta]}^{n,\varphi}$ in \eqref{eq.M_n_varepsilon} will diverge as $n\rightarrow\infty$. Therefore the limit (in distribution) of the measure valued process $\{U^{n,\varepsilon}\}_n$ does not exist when $\varepsilon>1/2$, and the central limit theorem for the Gaussian mixture approximation can only be possibly obtained when $\varepsilon\in(0,1/2]$.

As we will see in the following two sections, the essence of the analysis and proofs of the convergence in distribution is the same for different $\varepsilon$, except for some notational changes. In other words, the central limit theorem can be proven for all $\varepsilon\in(0,1/2]$ in the same manner, and the choice of $\varepsilon$ will not have a crucial influence on the proof. We therefore choose $\varepsilon=1/2$ in the remaining of the paper, since it gives us the optimal $L^2$-convergence rate ($1/n$) of the approximating algorithm.
Thus, with no risk of abuse of notations, we can eliminate the superscript $\varepsilon$ for $U^{n,\varepsilon}$, $\pi^{n,\varepsilon}$ and $\rho^{n,\varepsilon}$, and simply write them as $U^n$, $\pi^n$ and $\rho^n$ from next section to ease notations.

\section{Step One: Tightness}
In this section we prove the tightness of the measure-valued process $\{U_t^n;t\geq0\}$. It
is possible to obtain the tightness and convergence in distribution results
by endowing $\mathcal M_F{(\mathbb R)}$ with the weak topology. In this topology
a sequence of finite measures $\{\mu^n\}_{n\in\mathbb N}\subset\mathcal M_F(\mathbb
R)$ converges to $\mu\in\mathcal M_F(\mathbb R)$ if and only if for a set
$\mathcal S(\varphi)$ of test functions, $\mu^n(\varphi)$ converges to $\mu(\varphi)$
for all $\varphi\in\mathcal S(\varphi).$ $\mathcal S(\varphi)$ can be taken
to be $C_b^m(\mathbb R)$ for any $m\geq1$. 

Before proceeding further discussion on $U^n$, we define the metric on $\mathcal
M_F(\mathbb
R)$ which generates the weak topology. 
Let $\varphi_0=1$ and $\{\varphi_i\}_{i\geq0}$ be a sequence of functions
which are dense in the space of continuous functions with compact support
on $\mathbb R$. Then the metric $d_{\mathcal M}$ is defined as
$$
d_{\mathcal M}:\mathcal M_F(\mathbb R)\times\mathcal M_F(\mathbb R)\rightarrow[0,\infty),\qquad\quad
d_{\mathcal M}(\mu,\nu)=\sum_{i=0}^{\infty}\frac{\mu(\varphi_i)-\nu(\varphi_i)}{2^i\|\varphi_i\|_{0,\infty}};
$$
and $d_{\mathcal M}$ generates the weak topology on $\mathcal M_{F}(\mathbb
R)$ in the sense that $\mu^n$ converges weakly to $\mu$ if and only if $\lim_{n\rightarrow\infty}d_{\mathcal
M}(\mu^n,\mu)=0$ as $\{\varphi_i\}_{i\geq0}$ is a convergence determining
set of functions over $\mathcal M_F(\mathbb R)$.

However, the space $\left(D_{\mathcal M_F(\mathbb R)}[0,\infty),d_{\mathcal
M}\right)$
is separable but not complete under this metric because its underlying space
$\left(\mathcal M_F(\mathbb R),d_{\mathcal M}\right)$ is separable but not
complete.
This inconvenience makes us unable to make use of Prohorov's Theorem (see, for example, Theorem 2.4.7 in \cite{Karatzas
and Shreve}). In
order
to tackle this problem, we consider the one-point compactification of $\mathbb
R$
$$\overline{\mathbb R}\triangleq\mathbb R\cup\{\infty\},$$
Then we embed the space $D_{\mathcal M_F(\mathbb R)}[0,\infty)$
into the complete and separable space $D_{\mathcal M_F(\mathbb{\overline
R})}[0,\infty)$ by defining a map such that
$$
\mu\in\mathcal M_F({\mathbb R})\rightarrow\overline\mu\in\mathcal
M_F(\overline{\mathbb R})\quad\text{and}\quad\overline\mu(A)=\mu(A\cap\mathbb
R),\ \forall A\in\overline{\mathbb R}.
$$
The family $\{U_t^n\}$ can then be viewed as a stochastic process with sample
paths in the complete and separable space $D_{\mathcal M_F(\mathbb{\overline
R})}[0,\infty)$, or as a random variable with values in the space $\mathcal
P(D_{\mathcal M_F(\mathbb{\overline R})}[0,\infty))$ -- the space of probability
measures over $D_{\mathcal M_F(\mathbb{\overline R})}[0,\infty)$.

We are now ready to show that the family of processes $\{U^n\}$ is tight
on $[0,T]$ for all $T>0$. In other words, let $\{\tilde{\mathbb{P}}_n\}\subset\mathcal
P\left(D_{\mathcal M_F(\overline{\mathbb R})}[0,T]\right)$ be the family
of associated probability distributions of $U^n$; in other words, $\tilde{\mathbb
P}_n(B)=\tilde{\mathbb P}_n(U^n\in B)$ for all $B\in\mathcal B(D_{\mathcal
M_F(\overline{\mathbb
R})}[0,T])$. We aim to show that $\{\tilde{\mathbb P}_n\}$ is relatively
compact and hence, by Prohorov's Theorem, tight. To be specific, we will
make use of the following theorem (Theorem 2.1 in \cite{Roelly-Coppoletta}):
\begin{theorem}\label{thm.stochastic17}
A family of probabilities $\{\tilde{\mathbb{P}}_n\}_{n}\subset\mathcal
P\left(D_{\mathcal M_F(\overline{\mathbb R^d})}[0,T]\right)$ is tight, if
there exits a dense sequence $\{\tilde f_k\}_{k\geq0}$ in $C_b(\overline{\mathbb
R^d})$ such that for each $k\in\mathbb N$, $\{\pi_{\tilde f_k}\tilde{\mathbb
P}_n\}_n\subset\mathcal P\left(D_{\overline{\mathbb R}}[0,T]\right)$ is a
tight sequence of probabilities; where $\pi_{\tilde f_k}:\mathcal M_F(\overline{\mathbb
R^d})\rightarrow\overline{\mathbb R}$ is defined by $\pi_{\tilde f_k}(\mu)=\mu(\tilde
f_k)$ for $\mu\in\mathcal M_F(\overline{\mathbb R^d})$.
\end{theorem}

In the remaining of this section, because of the definition of the distance
$d_{\mathcal M}$, we choose $(\tilde f_k)_{k\geq0}$ to be defined as follows:
$\tilde f_0\equiv1$, and $\tilde f_k$ ($k\geq1$) is chosen so that $\tilde
f_k\big|_{\mathbb R}$ is a dense sequence in $\mathcal C_b^6(\mathbb R)$,
the space of six times differentiable continuous functions on $\mathbb R$,
vanishing at infinity with continuous partial derivatives up to and including
the sixth order.

According to Theorem \ref{thm.stochastic17}, it suffices to prove the tightness
result for $\{\pi_{\tilde f_k}\tilde{\mathbb P}_n\}_n$. We will make use
of the following criteria, which can be found in \cite{Ethier and Kurtz},
to show that $\{\pi_{\tilde f_k}U^n\}_n=\{U^n(\tilde f_k)\}_n$ is tight,
and then the tightness of $\{\pi_{\tilde f_k}\tilde{\mathbb P}_n\}$ follows
by applying Theorem \ref{thm.stochastic17}.

\begin{theorem}[Kurtz's criteria of relative compactness]\label{thm.Kurtz}
Let $(E,d)$ be a separable and complete metric space and let $\{X^n\}_{n\in\mathbb
N}$ be a sequence of processes with sample paths in $D_E[0,\infty)$. Suppose
that for every $\eta>0$ and rational $t$, there exists a compact set $\Gamma_{\eta,t}$
such that
\begin{equation}\label{eq.condition1}
\sup_n\mathbb P(X_t^n\notin\Gamma_{\eta,t})\leq\eta.
\end{equation}
Then $\{X^n\}_{n\in\mathbb N}$ is relatively compact if and only if the following
conditions hold:
\begin{itemize}
\item
For each $T'>0$, there exists $\zeta>0$ and a family $\{\gamma^n(\Delta):0<\Delta<1\}$
of non-negative random variables
\begin{equation}\label{eq.condition2}
\tilde{\mathbb E}\left[\left(1\wedge d(X_{t+u}^n,X_t^n)\right)^\zeta\left(1\wedge
d(X_{t}^n,X_{t-v}^n)\right)^\zeta|\mathcal F_t\right]\leq\tilde{\mathbb E}\left[\gamma^n(\Delta)|\mathcal
F_t\right]
\end{equation}
for $0\leq t\leq T'$, $0\leq u\leq\Delta$ and $0\leq v\leq\Delta\wedge t$;
\item
For $\gamma^n(\Delta)$, we have 
\begin{equation}\label{eq.condition3}
\lim_{\Delta\rightarrow0}\limsup_{n\rightarrow\infty}\tilde{\mathbb E}\left[\gamma^n(\Delta)\right]=0;
\end{equation}
\item
At the initial time
\begin{equation}\label{eq.condition4}
\lim_{\Delta\rightarrow0}\limsup_{n\rightarrow\infty}\tilde{\mathbb E}\left[\left(1\wedge
d(X_{\Delta}^n,X_0^n)\right)^\zeta\right]=0.
\end{equation}
\end{itemize}
\end{theorem}\
To justify \eqref{eq.condition1}, we need to prove the following lemma:
\begin{lemma}\label{lem.needed_for_tightness}
For all $\eta>0$, there exists a constant $\bar\beta$ such that for the associated
probabilities $\{\pi_{\tilde f_k}\tilde{\mathbb P}_n\}$ of $\{\pi_{\tilde
f_k}U^n\}$ and $A=\{x\in D_{\overline{\mathbb R}}[0,T]:\sup_{t\in[0,T]}|x(t)|>\bar\beta\}$,
we have
\begin{equation}
\pi_{\tilde f_k}\tilde{\mathbb P}_n(A)\leq\eta.
\end{equation}
\end{lemma}
\begin{proof}
Note that $\pi_{\tilde f_k}U_t^n=U_t^n(\tilde f_k)$, so that
\begin{align}
\pi_{\tilde f_k}\tilde{\mathbb P}_n(A)&=\tilde{\mathbb P}_n\pi_{\tilde f_k}^{-1}(A)
=\tilde{\mathbb P}_n\left(U^n\in D_{\mathcal M_F}[0,T]:\sup_{t}|U_t^n(\tilde
f_k)|>\bar\beta\right)\nonumber\\
&=\tilde{\mathbb P}_n\left(U^n\in D_{\mathcal M_F}[0,T]:\sup_{t}|\sqrt{n}(\rho_t^n(\tilde
f_k)-\rho_t(\tilde f_k))|>\bar\beta\right)
\leq\frac{\Lambda_T^n(\tilde f_k)}{\bar\beta^2},
\end{align}
where $\Lambda_T^n(\tilde f_k)=\tilde{\mathbb E}\left[\sup_t\left(\sqrt n(\rho_t^n(\tilde
f_k)-\rho_t(\tilde f_k))\right)^2\right]$.

It suffices to show that $\Lambda_T^n(\tilde f_k)$ is bounded above by a
constant independent of $n$, which is an immediate consequence of Jensen's
inequality and Theorem 4.18 in \cite{Crisan and Li 2}. Then we choose
$$\bar\beta^2=\frac{\eta}{\Lambda_T^n(\tilde f_k)}$$
and the proof is complete.
\end{proof}

In order to prove the tightness of $\{U^n(\tilde f_k)\}_n$, we need to
show that $\{U^n(\tilde f_k)\}_n$ satisfies \eqref{eq.condition2}, \eqref{eq.condition3} and \eqref{eq.condition4}. We prove these by showing that each
of the increments of the process appearing on the right hand side of \eqref{eq.evolution_for_U^n}
satisfies similar bounds.

In the following we will choose $\Delta$ to be sufficiently small. To be
specific, we let $\Delta<\frac{\delta}{2}$, where $\delta$ is the time length
between two resampling events. This ensures that either $[t-\Delta,t]$ or
$[t,t+\Delta]$ does not contain a resampling event, in other words, there
is at most one resampling event in $[t,t+u]$ and $[t-v,t]$, where $0\leq
u\leq\Delta$ and $0\leq v\leq\Delta\wedge t$.

If the resampling happens only in the interval $[t-v,t]$, and obtain
\begin{align}
&\tilde{\mathbb E}\left[\left(1\wedge d(X_{t+u}^n,X_t^n)\right)^\zeta\left(1\wedge
d(X_{t}^n,X_{t-v}^n)\right)^\zeta|\mathcal F_t\right]\nonumber
\leq\tilde{\mathbb E}\left[\left(1\wedge d(X_{t+u}^n,X_t^n)\right)^\zeta|\mathcal
F_t\right].
\end{align}

Therefore in order to determine $\gamma^n(\Delta)$ and shows that \eqref{eq.condition2}
is satisfied by $\{U^n(\tilde f_k)\}_n$, it suffices to find an appropriate
$\gamma^n(\Delta)$ for $\zeta=2$ and show that
\begin{equation}
\tilde{\mathbb E}\left[\left(1\wedge d(U^n_{t+u}(\tilde f_k),U^n_t(\tilde
f_k)\right)^2|\mathcal
F_t\right]\leq\tilde{\mathbb E}\left[\gamma^n(\Delta)|\mathcal
F_t\right].
\end{equation}
This will be done in the following proposition.
\begin{proposition}\label{prop.verify_condition2}
Let $k\in\mathbb N$, and we further assume that $\tilde f_k\in C_b^6(\mathbb
R)$, and Assumption (A) holds. Let the length between two resampling events
$\delta$ be fixed and let $\alpha\propto\frac{1}{\sqrt n}$.
Define the family $\{\gamma_u^n(\Delta):0<\Delta<1\}$
of non-negative random variables
\begin{align}\label{eq.gamma_n_Delta}
\gamma^n(\Delta)\triangleq3&n\Delta^2\sup_{s\in[t,t+u]}\left(\rho_s^n(A\tilde
f_k)-\rho_s(A\tilde f_k)\right)^2+3n\Delta\sup_{s\in[t,t+u]}\left(\rho_s^n(h\tilde
f_k)-\rho_s(h\tilde f_k)\right)^2\nonumber\\
+&\frac{3\Delta}{n}C_{\gamma}\|\tilde f_k\|_{6,\infty}^2\sum_{j=1}^n\sup_{s\in[t,t+u]}\Big(\xi_{i\delta}^na_j^n(s)
\Big)^2,
\end{align}
where $C_\gamma$ is a constant independent of $n$. By Theorem 4.18 in \cite{Crisan and Li 2},
we know that 
$$\sup_{s\in[t,t+u]}n\left(\rho_s^n(A\tilde
f_k)-\rho_s(A\tilde f_k)\right)^2\ \text{and}\ \sup_{s\in[t,t+u]}n\left(\rho_s^n(h\tilde
f_k)-\rho_s(h\tilde f_k)\right)^2$$ are bounded and independent of $\Delta$.
Then we have
\begin{equation}\label{eq.bound_for_gamma_u}
\tilde{\mathbb E}\left[1\wedge d(U^n_{t+u}(\tilde f_k),U^n_t(\tilde f_k))^2|\mathcal
F_t\right]\leq\tilde{\mathbb
E}\left[\gamma^n(\Delta)|\mathcal F_t\right].
\end{equation}
\end{proposition}
\begin{proof}
Bearing in mind that there is no resampling event within $[t,t+u]$, thus
$[(t+u)/\delta]=[t/\delta]$ and
$$M_{[(t+u)/\delta]}^{n,\tilde f_k}-M_{[t/\delta]}^{n,\tilde f_k}=0.$$ 
Therefor we have that
\begin{align}\label{eq.bound_for_each_term}
\tilde{\mathbb{E}}&\left[1\wedge d(U^n_{t+u}(\tilde f_k),U^n_t(\tilde f_k))^2\big|\mathcal
F_t\right]
\leq\tilde{\mathbb E}\left[|U^n_{t+u}(\tilde f_k)-U^n_t(\tilde f_k)|^2\big
|\mathcal F_t\right]\nonumber\\
=\tilde{\mathbb E}&\left[\left|\sqrt n\left(\left(\rho^n_{t+u}(\tilde f_k)-\rho_{t+u}(\tilde
f_k)\right)-\left(\rho^n_t(\tilde f_k)-\rho_t(\tilde f_k)\right)\right)\right|^2\Bigg|\mathcal
F_t\right]\nonumber\\
\leq3 n&\Bigg\{\tilde{\mathbb E}\left[\left(\int_t^{t+u}(\rho_s^n(A\tilde
f_k)-\rho_s(A\tilde f_k))ds\right)^2\Bigg |\mathcal F_t\right]
+\tilde{\mathbb E}\left[\left(\int_t^{t+u}(\rho_s^n(h\tilde f_k)-\rho_s(h\tilde
f_k))dY_s\right)^2\Bigg |\mathcal F_t\right]\nonumber\\
&+\frac{1}{n^2}\tilde{\mathbb E}\left[\left(\sum_{j=1}^n\int_t^{t+u}\xi_{i\delta}^na_j^n(s)\Big[
R_{s,j}^1(\tilde f_k)ds+R_{s,j}^2(\tilde f_k)dY_s+R_{s,j}^3(\tilde f_k)dV_s^{(j)}\Big]\right)^2\Bigg|\mathcal
F_t\right]\Bigg\}.
\end{align}
We examine each of the terms in \eqref{eq.bound_for_each_term} and observe
the following:\\
For the first term in \eqref{eq.bound_for_each_term}, by Jensen's inequality,
we have
\begin{align}\label{eq.bound1_for_U}
&\tilde{\mathbb E}\left[\left(\sqrt n\int_t^{t+u}(\rho_s^n(A\tilde
f_k)-\rho_s(A\tilde f_k))ds\right)^2\Bigg |\mathcal F_t\right]
\leq nu^2\sup_{s\in[t,t+u]}\tilde{\mathbb E}\left[\left(\rho_s^n(A\tilde
f_k)-\rho_s(A\tilde f_k\right)^2\Big|\mathcal F_t\right].
\end{align}
For the second term in \eqref{eq.bound_for_each_term}, 
\begin{align}\label{eq.bound2_for_U}
&\tilde{\mathbb E}\left[\left(\int_t^{t+u}\sqrt n\left(\rho_s^n(h\tilde f_k)-\rho_s(h\tilde
f_k)\right)dY_s\right)^2\Bigg |\mathcal F_t\right]
\leq un\sup_{s\in[t,t+u]}\tilde{\mathbb E}\left[\left(\rho_s^n(h\tilde
f_k)-\rho_s(h\tilde f_k)\right)^2\Big|\mathcal F_t\right].
\end{align}

For the remaining terms in \eqref{eq.bound_for_each_term}, note that 
$$R_{s,j}^1(\tilde f_k)\leq C_1\alpha\delta\|\tilde f_k\|_{6,\infty}\leq
\frac{C_1}{n}\|\tilde f_k\|_{6,\infty},$$
we then have
\begin{align}\label{eq.bound4_for_U}
&n\frac{1}{n^2}\tilde{\mathbb E}\left[\left(\sum_{j=1}^n\int_t^{t+u}\xi_{i\delta}^na_j^n(s)\Big[
R_{s,j}^1(\tilde f_k)ds\Big]\right)^2\Bigg|\mathcal
F_t\right]
\leq u\frac{C_1^2}{n}\|\tilde f_k\|_{6,\infty}^2\sum_{j=1}^n\sup_{s\in[t,t+u]}\tilde{\mathbb
E}\left[\left(\xi_{i\delta}^na_j^n(s)
\right)^2\Big|\mathcal F_t\right];
\end{align}
and also note that
$$R_{s,j}^2(\tilde f_k)\leq C_2\alpha\delta\|\tilde f_k\|_{4,\infty}\leq\frac{C_2}{n}\|\tilde
f_k\|_{4,\infty},$$
then we have
\begin{align}\label{eq.bound5_for_U}
&n\frac{1}{n^2}\tilde{\mathbb E}\left[\left(\sum_{j=1}^n\int_t^{t+u}\xi_{i\delta}^na_j^n(s)R_{s,j}^2(\varphi)dY_s\right)^2\Bigg|\mathcal
F_t\right]
\leq u\frac{C_2^2}{n}\|\tilde f_k\|_{4,\infty}^2\sum_{j=1}^n\sup_{s\in[t,t+u]}\tilde{\mathbb
E}\left[\Big(\xi_{i\delta}^na_j^n(s)\Big)^2\Big|\mathcal
F_t\right];
\end{align}

and finally since
$$R_{s,j}^3(\tilde f_k)\leq(C_0+C_3\alpha\delta)\|\tilde f_k\|_{5,\infty}\leq(C_0+C_3)\|\tilde
f_k\|_{5,\infty},$$
we have that
\begin{align}\label{eq.bound6_for_U}
&n\frac{1}{n^2}\tilde{\mathbb E}\left[\left(\sum_{j=1}^n\int_t^{t+u}\xi_{i\delta}^na_j^n(s)
R_{s,j}^3(\varphi)dV_s^{(j)}\right)^2\Bigg|\mathcal F_t\right]
\leq\frac{u}{n}(C_0+C_3)^2\|\tilde f_k\|_{5,\infty}^2\sum_{j=1}^n\sup_{s\in[t,t+u]}\tilde{\mathbb
E}\left[\Big(\xi_{i\delta}^na_j^n(s)\Big)^2\Big|\mathcal F_t\right].
\end{align}
Therefore, considering the bounds in the right hand sides of \eqref{eq.bound1_for_U},
\eqref{eq.bound2_for_U}, \eqref{eq.bound4_for_U}, \eqref{eq.bound5_for_U},
and \eqref{eq.bound6_for_U}; we can define $\gamma^n(\Delta)$ as in \eqref{eq.gamma_n_Delta}
by letting
$$C_{\gamma}=C_1^2+C_2^2+(C_0+C_3)^2.$$ By virtue of \eqref{eq.bound_for_each_term},
we know that \eqref{eq.bound_for_gamma_u} is satisfied.
\end{proof}

The above discussion defines $\gamma^n(\Delta)$ and shows that \eqref{eq.condition2}
is satisfied for $\{U^n(\tilde f_k)\}_n$.
The following proposition shows that $\gamma^n(\Delta)$ defined in \eqref{eq.gamma_n_Delta}
satisfies \eqref{eq.condition3}.
\begin{proposition}\label{prop.verify_condition3}
$\gamma^n(\Delta)$ defined in \eqref{eq.gamma_n_Delta} has the following
property
\begin{equation}\label{eq.limit_of_gamma_u_n}
\lim_{\Delta\rightarrow0}\limsup_{n\rightarrow\infty}\tilde{\mathbb E}\left[\gamma^n(\Delta)\right]
=0.
\end{equation}
\end{proposition}
\begin{proof}
We show this by looking at the expectation of each term in \eqref{eq.gamma_n_Delta}.\\
For the first term, by Theorem 4.18 in \cite{Crisan and Li 2} 
\begin{align}
&\tilde{\mathbb E}\left[n\Delta^2\sup_{s\in[t,t+u]}\left(\rho_s^n(A\tilde
f_k)-\rho_s(A\tilde f_k\right)^2\right]
\leq\Delta^2c^T\left\|A\tilde f_k\right\|_{m+2,\infty}^2\rightarrow0,\quad\text{as}\
\Delta\rightarrow0.\nonumber
\end{align}
Similarly, for the second term, 
\begin{align}
\tilde{\mathbb E}\left[n\Delta\sup_{s\in[t,t+u]}\left(\rho_s^n(h\tilde
f_k)-\rho_s(h\tilde f_k)\right)^2\right]\leq\Delta\tilde c^T\left\|h\tilde
f_k\right\|_{m+2,\infty}^2\rightarrow0,\quad\text{as}\ \Delta\rightarrow0.\nonumber
\end{align}

For the remaining term, again note that $(\alpha\delta)^2\sim1/n$, and 
\begin{align}
&\tilde{\mathbb E}\left[\sum_{j=1}^n\sup_{s\in[t,t+u]}\Big(\xi_{i\delta}^na_j^n(s)\Big)^2\right]
=\sum_{j=1}^n\tilde{\mathbb
E}\left[\sup_{s\in[t,t+u]}\Big(\xi_{i\delta}^na_j^n(s)\Big)^2\right]
\leq nc_2^{t,2}.\nonumber
\end{align}
Thus
\begin{align}
&\frac{\Delta}{n}C_{\gamma^n}\|\tilde
f_k\|_{6,\infty}^2\sum_{j=1}^n\tilde{\mathbb E}\left[\sup_{s\in[t,t+u]}\Big(\xi_{i\delta}^na_j^n(s)\Big)^2\right]
\leq\frac{\Delta}{n}\|\tilde
f_k\|_{6,\infty}^2nc_2^{t,2}=\Delta c_2^{t,2}\|\tilde f_k\|_{6,\infty}^2\rightarrow0,\quad\text{as}\
\Delta\rightarrow0.\nonumber
\end{align}
This completes the proof.
\end{proof}

The following proposition shows that \eqref{eq.condition4} holds for $\{U^n(\tilde
f_k)\}$.
\begin{proposition}\label{prop.verify_condition4}
For each $k\in\mathbb N$, we have 
\begin{equation}
\lim_{\Delta\rightarrow0}\limsup_{n\rightarrow\infty}\tilde{\mathbb E}\left[\left(1\wedge
d(U_{\Delta}^n(\tilde f_k),U_0^n(\tilde f_k))\right)^2\right]=0.
\end{equation}
\end{proposition}
\begin{proof}
The result follows immediately by continuity of $\{U^n(\tilde f_k)\}_n$ at
the initial time 0.
\end{proof}

\begin{theorem}
The measure-valued processes $\{U_t^n:t\in[0,T]\}_{n\geq1}$ forms a tight
sequence.
\end{theorem}
\begin{proof}
Lemma \ref{lem.needed_for_tightness}, Propositions \ref{prop.verify_condition2}
-- \ref{prop.verify_condition4}
state that all the conditions in Theorem \ref{thm.Kurtz} are satisfied. Then
by Theorem \ref{thm.Kurtz} we know that $\{\pi_{\tilde f_k}U^n\}_n$ is tight,
which implies that $\{\pi_{\tilde f_k}\tilde{\mathbb P}_n\}_n$ forms a tight
sequence on $\mathcal P(D_{\mathbb R}[0,T])$; then by Theorem \ref{thm.stochastic17}
we know $\{\tilde{\mathbb P}_n\}$ forms a tight sequence on $\mathcal P(D_{\mathcal
M_F(\mathbb R^d)}[0,T])$. By definition we can then conclude the following
tightness result.
\end{proof}

\begin{remark}
If we assume that the resampling happens only in $[t,t+u]$, then by exactly
the same discussion as above (except that we replace $s\in[t,t+u]$ by $s\in[t-v,u]$),
we can also obtain the tightness for the process $\{U_t^n\}_{n\geq1}$.  
\end{remark}

\section{Step Two: Limits of Convergent Subsequences}

In this section we show that $\{U^n\}_n$ converges in distribution to a uniquely
determined process $U$. The strategy of the proof of the convergence in distribution
is as follows: Since the sequence of the measure-valued process $\{U^n\}_n$
is tight, then any subsequence $\{U^{n_k}\}_k$ of $\{U^n\}_n$ contains a
convergent sub-subsequence $\{U^{n_{k_l}}\}_l$. We will prove that any convergent
subsequence has a weak limit $U$ which is the unique solution of \eqref{eq.evolution_of_U}.
This ensures that the entire sequence $\{U^n\}_n$ is convergent and its weak
limit is the solution $U$ of \eqref{eq.evolution_of_U}.

We need the following preliminary result.
\begin{lemma}\label{lem.preliminary}
Let $\varphi\in C_b^m(\overline{\mathbb R})$ ($m\geq6$) be a test function,
and define the measure-valued
processes
\begin{align}
\tilde\rho_t^{n.1}&\triangleq\frac{1}{n}\sum_{j=1}^n\xi_{i\delta}^na_j^n(t)\delta_{v_j^n(t)}=\sum_{j=1}^n\xi_t^n\bar
a_j^n(t)\delta_{v_j^n(t)},\nonumber\\
\tilde\rho_t^{n.2}&\triangleq\frac{1}{n}\sum_{j=1}^n\left\{\xi_{i\delta}^na_j^n(t)\right\}^2\delta_{v_j^n(t)}=
n\sum_{j=1}^n\left\{\xi_t^n\bar a_j^n(t)\right\}^2\delta_{v_j^n(t)}.
\end{align}
then for any $t\in[0,T]$,
$$\tilde\rho_t^{n,1}\rightarrow\tilde\rho_t^1,\quad\tilde\rho_t^{n,2}\rightarrow\tilde\rho_t^2,\quad\quad\tilde{\mathbb
P}-a.s.,$$
where $\tilde\rho^1$ is the solution of the Zakai equation, and $\tilde\rho^2$ is the measure-valued process
satisfying, for any $\varphi\in\mathcal D(A)$,
\begin{align}
\tilde\rho_t^2(\varphi)=\pi_0(\varphi)&+\int_0^t\Big\{\rho_s(\mathbf 1)\rho_s(A\varphi)
+\rho_s(h)\rho_s(h\varphi)\Big\}ds
+\int_0^t\Big\{\rho_s(\mathbf 1)\rho_s(h\varphi)+\rho_s(h)\rho_s(\varphi)\Big\}dY_s.
\end{align}
\end{lemma}
\begin{proof}
The proof is identical to that of $\rho_t^n$ converging to $\rho_t$, which is included in \cite{Crisan and Li 2}.
\end{proof}

\begin{proposition}\label{prop.evolution_of_U}
For any $\varphi\in C_b^6(\overline{\mathbb R})$, let $\Lambda^\varphi$ be
the process defined by
\begin{align}\label{eq.sum_of_limiting_processes}
\Lambda_t^\varphi=&\sum_{i=1}^{[t/\delta]}\rho_{i\delta}(\mathbf
1)\sqrt{\pi_{i\delta-}(\varphi^2)-\left(\pi_{i\delta-}(\varphi)\right)^2}\Upsilon_{i}+c_\omega\int_0^t\tilde\rho_s^1(\Psi\varphi)ds\nonumber\\
&+c_\omega\int_{0}^t\left(\tilde\rho_s^1(h\varphi''-(h\varphi)'')\right)dB_s^{(2)}+\int_0^t\sqrt{\tilde\rho_s^2\left((\sigma\varphi')^2\right)}dB_s^{(3)}
\end{align} 
for $t\in[0,T]$. In \eqref{eq.sum_of_limiting_processes}, 
$\{\Upsilon_i\}_{i\in\mathbb N}$ is a sequence of independent identically
distributed, standard normal random variables, and $\left\{\sqrt{\pi_{i\delta-}(\varphi^2)-\left(\pi_{i\delta-}(\varphi)\right)^2}\Upsilon_{i}\right\}_i$
are mutually independent given the $\sigma$-algebra $\mathcal Y$.
$c_\omega$ is a constant independent of $n$, and the operator $\Psi$ is defined
by
$$\Psi\varphi=\frac{f\varphi'''}{2}+\frac{\sigma\varphi^{(4)}}{4}-\frac{3(A\varphi)''}{2}.$$
$B^{(2)}$ and $B^{(3)}$ are two independent standard Brownian motion
both independent of the observation $Y$.

If $U$ is a $\mathcal D_{\mathcal M_F(\overline{\mathbb R})}[0,\infty)$-valued
process such that for $\varphi\in\mathcal C_b^6(\overline{\mathbb R})$
\begin{equation}\label{eq.evolution_of_U}
U_t(\varphi)=U_0(\varphi)+\int_{0}^tU_s(A\varphi)ds+\int_{0}^tU_s(h\varphi)dY_s+\Lambda_t^\varphi,
\end{equation}
then $U$ is pathwise unique. That is, for any two strong solutions $U^1$ and $U^2$ of \eqref{eq.evolution_of_U} with common
initial value $\mathbb P\left[U_{0}^1=U_{0}^2\right]=1$, the two processes
are indistinguishable, i.e. $\mathbb P\left[U_t^1=U_t^2;t\in[0,T]\right]=1$.
\end{proposition}
\begin{proof}
The argument here is similar to Theorem 2.21 and Remark 3.4 in \cite{Lucic and Heunis}.
Firstly, it can be seen that the first, third and fourth terms of \eqref{eq.sum_of_limiting_processes}
are martingales while the second term is not a martingale.

Suppose there exist two solutions $U^1$ and $U^2$ of \eqref{eq.evolution_of_U}.
Then take $\varphi\in C_b^6(\mathbb R)$, we have
\begin{equation}
U_t^i(\varphi)=U_0^i(\varphi)+\int_{0}^tU_s^i(A\varphi)ds+\int_{0}^tU_s^i(h\varphi)dY_s+\Lambda_t^\varphi,\quad i=1,2.
\end{equation}

For $i,j=\{1,2\}$ let $U^{ij}(\varphi_1,\varphi_2)\triangleq\tilde{\mathbb
E}[U^i(\varphi_1)U^j(\varphi_2)]$, for $\varphi_1,\varphi_2\in C_b^6(\overline{\mathbb
R})$. \\

By It\^o's formula we have
\begin{align}
U^{12}(\varphi_1,\varphi_2)=&\int_{0}^tU^{12}(\varphi_1,A\varphi_2)ds+\int_{0}^t
U^{12}(A\varphi_1,\varphi_2)ds+\int_{0}^t U^{12}(h\varphi_1,h\varphi_2)ds\nonumber\\
+&\int_{0}^t\tilde{\mathbb E}\Big[U_s^1(\varphi_1)\tilde\rho_s^1(\Psi\varphi_2)+{U_s^2(\varphi_2)\tilde\rho_s^1(\Psi\varphi_1)}\Big]ds\nonumber\\
+&\int_0^t\tilde{\mathbb E}\left[\sqrt{\tilde\rho_s^2((\sigma\varphi_1)^2)\tilde\rho_s^2((\sigma\varphi_2)^2)}+\tilde\rho_s^1(h\varphi_1''-(h\varphi_1)'')\tilde\rho_s^1(h\varphi_2''-(h\varphi_2)'')\right]ds\nonumber\\
+&\tilde{\mathbb E}\left[\sum_{i=0}^{[t/\delta]}\tilde{\mathbb E}\left[(\rho_{i\delta}(\mathbf
1))^2\left(\pi_{i\delta-}(\varphi_1\varphi_2)-\pi_{i\delta-}(\varphi_1)\pi_{i\delta-}(\varphi_2)\right)\big|\mathcal
F_{i\delta-}\right]\right];\nonumber
\end{align}

\begin{align}
 U^{11}(\varphi_1,\varphi_2)=&\int_{0}^t U^{11}(\varphi_1,A\varphi_2)ds+\int_{0}^t
U^{11}(A\varphi_1,\varphi_2)ds+\int_{0}^t U^{11}(h\varphi_1,h\varphi_2)ds\nonumber\\
+&\int_{0}^t\tilde{\mathbb E}\Big[U_s^1(\varphi_1)\tilde\rho_s^1(\Psi\varphi_2)+{U_s^1(\varphi_2)\tilde\rho_s^1(\Psi\varphi_1)}\Big]ds\nonumber\\
+&\int_0^t\tilde{\mathbb E}\left[\sqrt{\tilde\rho_s^2((\sigma\varphi_1)^2)\tilde\rho_s^2((\sigma\varphi_2)^2)}+\tilde\rho_s^1(h\varphi_1''-(h\varphi_1)'')\tilde\rho_s^1(h\varphi_2''-(h\varphi_2)'')\right]ds\nonumber\\
+&\tilde{\mathbb E}\left[\sum_{i=0}^{[t/\delta]}\tilde{\mathbb E}\left[(\rho_{i\delta}(\mathbf
1))^2\left(\pi_{i\delta-}(\varphi_1\varphi_2)-\pi_{i\delta-}(\varphi_1)\pi_{i\delta-}(\varphi_2)\right)\big|\mathcal
F_{i\delta-}\right]\right];\nonumber
\end{align}
and similarly for for $U^{21}(\varphi_1,\varphi_2)$ and $U^{22}(\varphi_1,\varphi_2)$.

Let
\begin{equation}
v_t=\left(U_t^{12}-U_t^{11}\right)+\left(U_t^{21}-U_t^{22}\right),
\end{equation}
it then follows that
\begin{align}\label{eq.evolution_of_difference}
v_t(\varphi_1,\varphi_2)
=\int_{0}^tv_s(\varphi_1,A\varphi_2)ds+\int_{0}^tv_s(A\varphi_1,\varphi_2)ds+\int_{0}^tv_s(h\varphi_1,h\varphi_2)ds;
\end{align}
and $v_0(\varphi_1,\varphi_2)=0$.

It follows by Theorem 2.21(i) and Remark 3.4 in \cite{Lucic and Heunis} that
\eqref{eq.evolution_of_difference} has a unique solution and since \eqref{eq.evolution_of_difference}
is a homogeneous equation beginning at $0$. Then we have $v_t(\varphi_1,\varphi_2)\equiv0$,
which
implies
$$\left(U_t^{11}-U_t^{12}\right)+\left(U_t^{22}-U_t^{21}\right)=0,$$
that is to say, for $\varphi_1=\varphi_2=\varphi$
\begin{align}
&\tilde{\mathbb E}\left[U_t^1(\varphi_1)U_t^1(\varphi)-U_t^1(\varphi_1)U_t^2(\varphi)\right]+\tilde{\mathbb
E}\left[U_t^2(\varphi_1)U_t^2(\varphi)-U_t^2(\varphi_1)U_t^1(\varphi)\right]
=\tilde{\mathbb E}\left[\left(U_t^1(\varphi)-U_t^2(\varphi)\right)^2\right]=0;\nonumber
\end{align}
and thus $U^1(\varphi)=U^2(\varphi)$ for $\varphi\in C_b^6(\overline{\mathbb
R})$, which in turn implies that the solution $U$ of \eqref{eq.evolution_of_U}
is unique (See Exercise 4.1 in \cite{Bain and Crisan}).
\end{proof}

The following Theorem \ref{thm.main_result} states that unique solution $\{U\}$
of \eqref{eq.evolution_of_U} is indeed the weak limit of any convergent subsequence of the measure-valued
process $\{U^n\}_n$, in other words, $\{U^n\}_n$ converges in distribution
to $\{U\}$.

\begin{theorem}\label{thm.main_result}
Under Assumption (A), any convergent subsequence  of $\{U^n\}_n$ has a limit $U$ in distribution that is the unique $\mathcal D_{\mathcal M_F(\overline{\mathbb
R})}[0,\infty)$-valued
process $U$ solving the following equation
\begin{equation}\label{eq.limit_of_U_n}
U_t(\varphi)=U_0(\varphi)+\int_0^tU_s(A\varphi)ds+\int_0^tU_s(h\varphi)dY_s+\Lambda_t^\varphi,
\end{equation}
for $\varphi\in C_b^6(\overline{\mathbb R})$, where $\Lambda_t^\varphi$ is defined as in \eqref{eq.sum_of_limiting_processes}.
\end{theorem}
\begin{proof}
From Proposition 5.3.20 in \cite{Karatzas and Shreve} and its extension to
stochastic partial differential equation and infinitely dimensional stochastic
differential equations, it follows that for solutions of stochastic partial
differential equations, pathwise uniqueness implies uniqueness in law. This
 was done by Ondrej\'at (see \cite{Ondrejat})
and R\"ockner, Schmuland and Zhang (see \cite{Rockner et al}). 

Thus by Proposition \ref{prop.evolution_of_U} the solution $U$ of \eqref{eq.evolution_of_U}
is unique in distribution.

Now let $\{U^{n_{k}}\}_k$ be any convergent (in distribution) subsequence
of $\{U^n\}_n$ to a process $U$. We then verify that this process $U$ solves
\eqref{eq.evolution_of_U}, and then the uniqueness of solution of \eqref{eq.evolution_of_U}
implies that the original sequence $\{U^n\}_n$ converges to $U$ as well.
Bearing in mind that $U^{n_{k}}$ satisfies \eqref{eq.evolution_for_U^n},
it then essentially suffices to show that $\Lambda_t^\varphi$ in \eqref{eq.limit_of_U_n},
which is given by the weak limits of $\sqrt nM_{[t/\delta]}^{n,\varphi}$
and $\sqrt nB_t^{n,\varphi}$ in \eqref{eq.evolution_for_U^n}, does satisfy
\eqref{eq.sum_of_limiting_processes}.

We first denote by 
$$
\bar\Lambda_t^\varphi\triangleq\Lambda_t^\varphi-\int_0^t\tilde\rho_s^1(\Psi\varphi)ds
$$
the martingale part of $\Lambda_t^\varphi$. Then
we only need to show that $\bar\Lambda^\varphi$ has the quadratic variation
which is the same as that of $\Lambda^\varphi$ in \eqref{eq.sum_of_limiting_processes}.
In order to do so, we show that for all $d,d'\geq0$, $0\leq t_1<t_2<\cdots<t_d\leq
s\leq T$, $0\leq t_1'<t_2'<\cdots<t_{d'}'\leq t\leq T$, continuous bounded
functions $\alpha_1,\ldots,\alpha_d$ on $\mathcal M_F(\overline{\mathbb R})$
and continuous functions $\alpha_1',\ldots,\alpha_{d'}'$ on $\overline{\mathbb
R}$; we have:
\begin{equation}\label{eq.mean}
\tilde{\mathbb E}\left[\left(\bar\Lambda_t^\varphi-\bar\Lambda_s^\varphi\right)\prod_{i=1}^d\alpha_i(U_{t_i})\prod_{j=1}^{d'}\alpha_j'(Y_{t_j'})\right]=0,
\end{equation}
and
\begin{align}\label{eq.variance}
\tilde{\mathbb E}\Bigg[\Bigg((\bar \Lambda_t^\varphi&-\bar \Lambda_s^\varphi)^2-\int_s^t\Big\{\tilde\rho_r^2\left((\sigma\varphi')^2\right)+\left(\tilde\rho_r^1\left(h\varphi''-(h\varphi)''\right)\right)^2\Big\}dr\nonumber\\
-&\sum_{i=[s/\delta]+1}^{[t/\delta]}\left(\rho_{i\delta}(\mathbf 1)\right)^2\left[\pi_{i\delta-}(\varphi^2)-\left(\pi_{i\delta-}(\varphi)\right)^2\right]\Bigg)\prod_{i=1}^d\alpha_i(U_{t_i})\prod_{j=1}^{d'}\alpha_j'(Y_{t_j'})\Bigg]=0.
\end{align}

To prove \eqref{eq.mean}, we first observe the following:
\begin{align}
\lim_{n\rightarrow\infty}\frac{1}{\sqrt {n}}\sum_{j=1}^n\int_{0}^{t}\xi_{[s/\delta]\delta}^na_j^n(s)R_{s,j}^1(\varphi)ds
\triangleq\lim_{n\rightarrow\infty}\Lambda_t^{n,R^1,\varphi}=\int_0^t\tilde\rho_s^1(\Psi\varphi)ds,
\end{align}
the proof can be found in \cite{Li}.
Then note that
$$
\bar\Lambda_t^\varphi-\bar\Lambda_s^\varphi=U_t(\varphi)-U_s(\varphi)-\int_s^tU_r(A\varphi)dr-\int_s^tU_r(h\varphi)dY_r-\int_s^t\tilde\rho_r^1
\left(\Psi\varphi\right)dr,
$$
thus showing \eqref{eq.mean} is equivalent to showing
\begin{align}\label{eq.mean_equivalence}
\tilde{\mathbb E}\Bigg[\Bigg(&U_t(\varphi)-U_s(\varphi)-\int_s^tU_r(A\varphi)dr-\int_s^tU_r(h\varphi)dY_r-\int_s^t\tilde\rho_r^1(\Psi\varphi)dr\Bigg)
\times\prod_{i=1}^d\alpha_i(U_{t_i})\prod_{j=1}^{d'}\alpha_j'(Y_{t_j'})\Bigg]=0.
\end{align}
This equality will follow by virtue of the martingale property of $\bar\Lambda_t^\varphi-\bar\Lambda_s^\varphi$.\\

By virtue of the existence of $\Lambda_T^n(\tilde f_k)$ in Lemma \ref{lem.needed_for_tightness},
it follows , for $n'\in\mathbb N$, that
$$
\sup_{n'}\tilde{\mathbb E}\left[(U^{n'}(\varphi))^2\right]<\infty,
$$
which implies that $\{U^{n_{k}}\}$ is uniformly integrable (see II.20, Lemma
20.5 in \cite{Rogers and Williams}). Therefore we have that
\begin{align}
&\lim_{k\rightarrow\infty}\tilde{\mathbb E}\left[U_t^{n_{k}}(\varphi)\prod_{i=1}^d\alpha_i(U_{t_i}^{n_{k}})\prod_{j=1}^{d'}\alpha_j'(Y_{t_j'})\right]
=\tilde{\mathbb E}\left[U_t(\varphi)\prod_{i=1}^d\alpha_i(U_{t_i})\prod_{j=1}^{d'}\alpha_j'(Y_{t_j'})\right],\nonumber\\
&\lim_{k\rightarrow\infty}\tilde{\mathbb E}\left[U_s^{n_{k}}(\varphi)\prod_{i=1}^d\alpha_i(U_{t_i}^{n_{k}})\prod_{j=1}^{d'}\alpha_j'(Y_{t_j'})\right]
=\tilde{\mathbb E}\left[U_s(\varphi)\prod_{i=1}^d\alpha_i(U_{t_i})\prod_{j=1}^{d'}\alpha_j'(Y_{t_j'})\right].\nonumber
\end{align}

By Burkholder-Davis-Gundy inequality, we know that
$$
\sup_{n'}\tilde{\mathbb E}\left[\left(\int_0^tU_r^{n'}(A\varphi)dr\right)^2\right]<\infty;
$$
thus we have 
\begin{align}
\lim_{k\rightarrow\infty}\tilde{\mathbb E}\left[\int_s^tU_r^{n_{k}}(A\varphi)dr\prod_{i=1}^d\alpha_i(U_{t_i}^{n_{k}})\prod_{j=1}^{d'}\alpha_j'(Y_{t_j'})\right]
=\tilde{\mathbb E}\left[\int_s^tU_r(A\varphi)dr\prod_{i=1}^d\alpha_i(U_{t_i})\prod_{j=1}^{d'}\alpha_j'(Y_{t_j'})\right].\nonumber
\end{align}
Similarly, by Burkholder-Davis-Gundy inequality, we can show that
$$
\sup_{n'}\tilde{\mathbb E}\left[\left(\int_s^tU_r^{n'}(h\varphi)dY_r\right)^2\right]<\infty,
$$
we therefore have that (by Theorem 2.2 in \cite{Kurtz and Protter}), since
$(U^{n_{k}},Y)$ converges in distribution to $(U,Y)$, then $(U^{n_{k}},Y,\int_s^tU_r^{n_{k}}(h\varphi)dY_r)$
also converges in distribution to $(U,Y,\int_s^tU_r(h\varphi)dY_r)$, thus
we have
\begin{align}
\lim_{k\rightarrow\infty}\tilde{\mathbb E}\left[\int_s^tU_r^{n_{k}}(h\varphi)dY_r\prod_{i=1}^d\alpha_i(U_{t_i}^{n_{k}})\prod_{j=1}^{d'}\alpha_j'(Y_{t_j'})\right]
=\tilde{\mathbb E}\left[\int_s^tU_r(h\varphi)dY_r\prod_{i=1}^d\alpha_i(U_{t_i})\prod_{j=1}^{d'}\alpha_j'(Y_{t_j'})\right].\nonumber
\end{align}

For $\int_s^t\tilde\rho_r^1
\left(\Psi\varphi\right)dr$, we have
\begin{align}
\lim_{k\rightarrow\infty}\tilde{\mathbb E}\left[\Lambda_t^{{n_{k}},R^1,\varphi}\prod_{i=1}^d\alpha_i(U_{t_i}^{n_{k}})\prod_{j=1}^{d'}\alpha_j'(Y_{t_j'})\right]
=\tilde{\mathbb E}\left[\int_s^t\tilde\rho_r^1
\left(\Psi\varphi\right)dr\prod_{i=1}^d\alpha_i(U_{t_i})\prod_{j=1}^{d'}\alpha_j'(Y_{t_j'})\right].\nonumber
\end{align}
Now we have shown \eqref{eq.mean_equivalence}, and hence \eqref{eq.mean}.

In order to show the second equality \eqref{eq.variance}, we firstly make
the following observations about the limits of the terms in \eqref{eq.evolution_for_U^n}:
\begin{itemize}
\item
We have
\begin{equation}\label{eq.observation_one}
\lim_{n\rightarrow\infty}\left\langle\sqrt nA_.^{n,\varphi}\right\rangle_t=\sum_{i=1}^{[t/\delta]}\left(\rho_{i\delta}(\mathbf
1)\right)^2\left[\pi_{i\delta-}(\varphi^2)-\left(\pi_{i\delta-}(\varphi)\right)^2\right].
\end{equation}
If we let
\begin{align}\label{eq.branching}
\bar A_t^{\varphi}\triangleq\sum_{i=1}^{[t/\delta]}\rho_{i\delta}(\mathbf
1)\sqrt{\pi_{i\delta-}(\varphi^2)-\left(\pi_{i\delta-}(\varphi)\right)^2}\Upsilon_{i},
\end{align}
where $\{\Upsilon_i\}_{i\in\mathbb N}$ is a sequence of independent identically
distributed, standard normal random variables, and $\left\{\sqrt{\pi_{i\delta-}(\varphi^2)-\left(\pi_{i\delta-}(\varphi)\right)^2}\Upsilon_{i}\right\}_i$
are mutually independent given the $\sigma$-algebra $\mathcal Y$; then we
have
$
\langle\bar A_\cdot^{\varphi}\rangle_t=\lim_{n\rightarrow\infty}\left\langle\sqrt
nA_.^{n,\varphi}\right\rangle_t.
$
\item
For $G_{[t/\delta]}^{n,\varphi}$, we have
\begin{equation}\label{eq.observation_two}
\lim_{n\rightarrow\infty}\left|\sqrt nG_{[t/\delta]}^{n,\varphi}\right|=0\quad\text{a.s.}.
\end{equation}
\item
We have
\begin{align}\label{eq.observation_three}
\lim_{n\rightarrow\infty}\left\langle\frac{1}{\sqrt n}\sum_{j=1}^n\int_{0}^{\cdot}\xi_{[s/\delta]\delta}^na_j^n(s)R_{s,j}^2(\varphi)dY_s\right\rangle_t
\triangleq\lim_{n\rightarrow\infty}\left\langle\Lambda_\cdot^{n,R^2,\varphi}\right\rangle_t=\left\langle\Lambda_\cdot^{R^2,\varphi}\right\rangle_t,
\end{align}
where
\begin{equation}\label{eq.Lambda_R2}
\Lambda_t^{R^2,\varphi}=c_\omega\int_{0}^t\left(\tilde\rho_s^1(h\varphi''-(h\varphi)'')\right)dB_s^{(2)},
\end{equation}
$c_\omega$ is a constant and $B^{(2)}$ is a Brownian motion independent of
$Y$.
\item
We have that
\begin{align}\label{eq.observation_four}
\lim_{n\rightarrow\infty}\left\langle\frac{1}{\sqrt n}\sum_{j=1}^n\int_{0}^{\cdot}\xi_{[s/\delta]\delta}^na_j^n(s)R_{s,j}^3(\varphi)dV_s^{(j)}\right\rangle_t
\triangleq\lim_{n\rightarrow\infty}\left\langle\Lambda_\cdot^{n,R^3,\varphi}\right\rangle_t=\left\langle\Lambda_\cdot^{R^3,\varphi}\right\rangle_t,
\end{align}
where 
\begin{equation}
\Lambda_t^{R^3,\varphi}=\int_0^t\sqrt{\tilde\rho_s^2\left((\sigma\varphi')^2\right)}dB_s^{(3)},
\end{equation}
$B^{(3)}$ is a Brownian motion independent of $B^{(2)}$ and $Y$.
\end{itemize}
The proofs of these observations can be found in Appendix \ref{sec.proofs_of_quantities}.

From the above observations, we obtain that 
\begin{align}
&\tilde{\mathbb E}\Bigg[(\bar\Lambda_t^\varphi-\bar\Lambda_s^\varphi)^2\prod_{i=1}^d\alpha_i(U_{t_i})\prod_{j=1}^{d'}\alpha_j'(Y_{t_j'})\Bigg]\nonumber\\
=&\lim_{k\rightarrow\infty}\tilde{\mathbb E}\Bigg[\Bigg(\left(\left\langle\sqrt
nA_\cdot^{n_{k},\varphi}\right\rangle_t-\left\langle\sqrt
nA_\cdot^{n_{k},\varphi}\right\rangle_s\right)+\left(\left\langle\Lambda_\cdot^{n_{k},R^2,\varphi}\right\rangle_t-\left\langle\Lambda_\cdot^{n_{k},R^2,\varphi}\right\rangle_s\right)\nonumber\\
&\qquad\quad\quad+\left(\left\langle\Lambda_\cdot^{n_{k},R^3,\varphi}\right\rangle_t-\left\langle\Lambda_\cdot^{n_{k},R^3,\varphi}\right\rangle_s\right)\Bigg)\times\prod_{i=1}^d\alpha_i(U_{t_i}^{n_{k}})\prod_{j=1}^{d'}\alpha_j'(Y_{t_j'})\Bigg]\nonumber
\end{align}
\begin{align}
=&\lim_{k\rightarrow\infty}\tilde{\mathbb E}\Bigg[\Bigg(
\sum_{i=[s/\delta]+1}^{[t/\delta]}\left(\rho_{i\delta}^{n_{k}}(\mathbf 1)\right)^2\left[\pi_{i\delta-}^{n_{k}}(\varphi^2)-\left(\pi_{i\delta-}^{n_{k}}(\varphi)\right)^2\right]\nonumber\\
&\quad\quad\quad+\int_{s}^t\left(\tilde\rho_r^{n_{k},1}(h\varphi''-(h\varphi)'')\right)^2dr+\int_{s}^t{\tilde\rho_r^{n_{k},2}\left((\sigma\varphi')^2\right)}dr\Bigg)\times
\prod_{i=1}^d\alpha_i(U_{t_i}^{n_{k}})\prod_{j=1}^{d'}\alpha_j'(Y_{t_j'})\Bigg]\nonumber\\
=&\tilde{\mathbb E}\Bigg[\Bigg(
\sum_{i=[s/\delta]+1}^{[t/\delta]}\left(\rho_{i\delta}(\mathbf 1)\right)^2\left[\pi_{i\delta-}(\varphi^2)-\left(\pi_{i\delta-}(\varphi)\right)^2\right]\nonumber\\
&\qquad\quad\quad+\int_{s}^t\left(\tilde\rho_r^{1}(h\varphi''-(h\varphi)'')\right)^2dr+\int_{s}^t{\tilde\rho_r^{2}\left((\sigma\varphi')^2\right)}dr\Bigg)\times
\prod_{i=1}^d\alpha_i(U_{t_i})\prod_{j=1}^{d'}\alpha_j'(Y_{t_j'})\Bigg]\nonumber\\
=&\tilde{\mathbb E}\left[\left(\langle\bar\Lambda.^\varphi\rangle_t-\langle\bar\Lambda.^\varphi\rangle_s\right)\prod_{i=1}^d\alpha_i(U_{t_i})\prod_{j=1}^{d'}\alpha_j'(Y_{t_j'})\right];
\end{align}
and \eqref{eq.variance} follows from this identity.
\end{proof}

\begin{corollary}
Under Assumption (A), for and $t\geq0$ define
$
\bar U_t^n\triangleq\sqrt n\left(\pi_t^n-\pi_t\right).
$
Then $\{\bar U^n\}_n$ converges in distribution to a unique $D_{\mathcal
M_F(\overline{\mathbb R})}[0,\infty)$-valued process $\bar U=\{\bar U_t:t\geq0\}$,
such that, for any test function $\varphi\in C_b^6(\overline{\mathbb R})$,
\begin{equation}
\bar U_t(\varphi)=\frac{1}{\rho_t(\mathbf 1)}\left(U_t(\varphi)-\pi_t(\varphi)U_t(\mathbf1)\right),
\end{equation}
where $U$ satisfies \eqref{eq.evolution_of_U}.
\end{corollary}
\begin{proof}
By the fact that
$$
\pi_t^n(\varphi)-\pi_t(\varphi)=\frac{1}{\rho_t(\mathbf1)}(\rho_t^n(\varphi)-\rho_t(\varphi))-\frac{\pi_t^n(\varphi)}{\rho_t(\mathbf1)}(\rho_t^n(\mathbf1)-\rho_t(\mathbf1)),
$$
and $\rho_t^n(\varphi)\rightarrow\rho_t(\varphi),\ \text{a.s.}$ and $\pi_t^n(\varphi)\rightarrow\pi_t(\varphi)\
\text{a.s.}$ (see Remark \ref{rmk.almost_sure_convergence}),
we have the result.
\end{proof}

\begin{remark}
The central limit theorem in this paper is proven in Sections 4 and 5 with $\varepsilon=1/2$. However, it should be noted that the result also holds when $\varepsilon\in(0,1/2)$, and the corresponding proof is similar. Therefore the proof of the main result of the paper, Theorem \ref{thm.main_CLT}, is completed without additional arguments for different $\varepsilon$.
\end{remark}

\begin{remark}
In this chapter we view $\{U^n\}_{n\in\mathbb N}$ and its weak limit $\{U\}$
as processes with sample paths in $D_{\mathcal M_F(\overline{\mathbb R})}[0,\infty)$,
which is complete and separable. 
In fact, $U$ takes value in a smaller
space $\mathcal M_F(\mathbb R)$ (i.e. $U$ is a $D_{\mathcal M_F({\mathbb
R})}[0,\infty)$-valued random variable). In other words, $U$ has no mass
`escaping' to infinity.
This is shown by using the same approach as in Section 5 in \cite{Crisan 4}.

Since the weak topology on ${\mathcal M_F({\mathbb
R})}$ coincides with the trace topology from ${\mathcal M_F(\overline{\mathbb
R})}$ to ${\mathcal M_F({\mathbb R})}$, it follows that $U$ has sample paths
in $D_{\mathcal M_F({\mathbb R})}[0,\infty)$. It then suffices to show that
that for arbitrary $t$, there exists a sequence of compact sets $\{K_p\}_{p>0}\in\mathbb
R$ (possibly depending on t) which exhaust $\mathbb R$ such that for all
$\varepsilon>0$,
\begin{equation}\label{eq.compact_R}
\lim_{p\rightarrow\infty}\tilde{\mathbb P}\left[\sup_{s\in[0,t]}\left(U_s(\mathbf1_{K_p^c})\right)\geq\varepsilon\right]=0,
\end{equation}
where $K_p^c$ denotes the compliment of $K_p$. 
The proof of \eqref{eq.compact_R} can be found in Section 5 in \cite{Crisan
4}.
\end{remark}

\section{Conclusions}
In this paper, we analyse the Gaussian mixture approximations to the solution of the nonlinear filtering problem. In addition to the $L^2$-convergence result obtained in \cite{Crisan and Li 2}, we prove a central limit type theorem of the Gaussian mixture approximation, and find that the optimal value for the parameter $\varepsilon$, which measures the ``Gaussianity'' of the approximating system, is $1/2$. It can be seen that, asymptotically (as $n\rightarrow\infty$), the mean square error between the approximating measure and the true solution of the filtering problem is (roughly) of order $1/n$, and the recalibrated error converges in distribution to a unique measure-valued process.

It should also be noted that the central limit theorem obtained in this paper is based on the approximating system under which the Multinomial branching algorithm is chosen. It is also worth studying the central limit theorem for the approximating system under the Tree Based Branching Algorithm, and this is left as future work.

\appendix
\addappheadtotoc

\section{Appendix}

\subsection{Almost sure limits of $\pi^n$ and $\rho^n$}
\begin{lemma}\label{lem.asymptotics_of_omega}
If the approximation $\pi^n$ is defined by \eqref{eq.gaussian_mixture_approximation},
in other words,
$$
\pi_t^n(\varphi)=\sum_{j=1}^n\bar a_j^n(t)\int_{\mathbb R}\varphi\left(v_j^n(t)+y\sqrt{\omega_j^n(t)}\right)
\frac{1}{\sqrt{2\pi}}\exp\left(-\frac{y^2}{2}\right)dy;
$$ 
then we have
\begin{equation}
\pi_t(\varphi)=\lim_{n\rightarrow\infty}\pi_t^n(\varphi)=\lim_{n\rightarrow\infty}\sum_{j=1}^n\bar
a_j^n(t)\varphi(v_j^n(t)).
\end{equation}
That is, asymptotically, the variances of  the Gaussian measures do
not contribute to the approximation, and the combination of positions and
weights provide a good approximation. 
\end{lemma}
\begin{proof}
See Appendix B in \cite{Li}.
\end{proof}

As a direct consequence, we have the following corollary for the unnormalised
approximation $\rho^n$:
\begin{corollary}\label{coro.asymptotics_of_omega}
If the approximation $\rho^n$ is defined as
$$
\rho_t^n(\varphi)=\xi_t^n\pi_t^n(\varphi)=\xi_{t}^n\sum_{j=1}^n\bar a_j^n(t)\int_{\mathbb
R}\varphi\left(v_j^n(t)+y\sqrt{\omega_j^n(t)}\right)
\frac{1}{\sqrt{2\pi}}\exp\left(-\frac{y^2}{2}\right)dy;
$$
then we have
\begin{equation}
\rho_t(\varphi)=\lim_{n\rightarrow\infty}\rho_t^n(\varphi)=\lim_{n\rightarrow\infty}\xi_t^n\sum_{j=1}^n\bar
a_j^n(t)\varphi(v_j^n(t)).
\end{equation}
\end{corollary}

\begin{remark}\label{rmk.almost_sure_convergence}
By Lemma \ref{lem.asymptotics_of_omega} we know asymptotically as $n\rightarrow\infty$,
the Gaussian mixture approximation performs just as good as the classic particle
filters. Furthermore, from Chapter 8 in \cite{Bain and Crisan} and Lemma
\ref{lem.asymptotics_of_omega}, we know that
$$
\rho_t^n(\varphi)\rightarrow\rho_t(\varphi)\quad\text{and}\quad\pi_t^n(\varphi)\rightarrow\pi_t(\varphi)\quad\text{almost
surely}.
$$
\end{remark}

\subsection{Proof of \eqref{eq.observation_one}, \eqref{eq.observation_two}, \eqref{eq.observation_three}, and \eqref{eq.observation_four}}\label{sec.proofs_of_quantities}
\begin{lemma}[\eqref{eq.observation_one}]
Assume the conditions in Proposition \ref{prop.verify_condition2} hold, then
\begin{equation}
\lim_{n\rightarrow\infty}\left\langle\sqrt nA_.^{n,\varphi}\right\rangle_t=\sum_{i=1}^{[t/\delta]}\left(\rho_{i\delta}(\mathbf
1)\right)^2\left[\pi_{i\delta-}(\varphi^2)-\left(\pi_{i\delta-}(\varphi)\right)^2\right].
\end{equation}
If we let
\begin{align}
\bar A_t^{\varphi}\triangleq\sum_{i=1}^{[t/\delta]}\rho_{i\delta}(\mathbf
1)\sqrt{\pi_{i\delta-}(\varphi^2)-\left(\pi_{i\delta-}(\varphi)\right)^2}\Upsilon_{i},
\end{align}
where $\{\Upsilon_i\}_{i\in\mathbb N}$ is a sequence of independent identically
distributed, standard normal random variables, and $\left\{\sqrt{\pi_{i\delta-}(\varphi^2)-\left(\pi_{i\delta-}(\varphi)\right)^2}\Upsilon_{i}\right\}_i$
are mutually independent given the $\sigma$-algebra $\mathcal Y$; then we
have
$
\langle\bar A_\cdot^{\varphi}\rangle_t=\lim_{n\rightarrow\infty}\left\langle\sqrt
nA_.^{n,\varphi}\right\rangle_t.
$
\end{lemma}
\begin{proof}
Note that $A^{n,\varphi}$ is a discrete time martingale, then
\begin{align}
\lim_n\left\langle\sqrt nA_.^{n,\varphi}\right\rangle_t
=&\lim_n\sum_{i=1}^{[t/\delta]}(\rho_{i\delta}^n(\mathbf 1))^2\left[\sum_{j=1}^n\bar
a_{j}^n(i\delta-)\left(\varphi(X_{j}^n(i\delta))\right)^2-\left(\sum_{j=1}^n\bar
a_{j}^n(i\delta-)\varphi(X_{j}^n(i\delta))\right)^2\right]\nonumber\\
=&\sum_{i=1}^{[t/\delta]}\left(\rho_{i\delta}(\mathbf 1)\right)^2\left[\pi_{i\delta-}(\varphi^2)-\left(\pi_{i\delta-}(\varphi)\right)^2\right],\nonumber
\end{align}
here we made use of Lemma \ref{lem.asymptotics_of_omega} and Remark \ref{rmk.almost_sure_convergence}.

The second part of the lemma is obvious. 
\end{proof}

\begin{lemma}[\eqref{eq.observation_two}]
Assume the conditions in Proposition \ref{prop.verify_condition2} hold, then
\begin{equation}
\lim_{n\rightarrow\infty}\left|\sqrt nG_{[t/\delta]}^{n,\varphi}\right|=0\quad\text{a.s.}.
\end{equation}
\end{lemma}
\begin{proof}
For $G^{n,\varphi}$, we know that
\begin{align}
\sqrt nG_{[t/\delta]}^{n,\varphi}
=&\sum_{i=1}^{[t/\delta]}\sum_{j=1}^n\sqrt n\xi_{i\delta}^n\bar a_{j}^n(i\delta-)\left[\varphi(X_{j}^n(i\delta))-\tilde{\mathbb{E}}\left(\varphi(X_{j}^n(i\delta))\right)\right],\nonumber
\end{align}
first note that $X_{j}^n(i\delta)\sim
N\left(v_{j}^n(i\delta),\omega_{j}^n(i\delta)\right)$ and $X_{j}^n$s are
mutually independent $(j=1,\ldots,n)$, also not the fact that $\omega\sim\mathcal
O(1/\sqrt n)$; if we let $Z_j^n(i\delta)\triangleq X_{j}^n(i\delta)-\tilde{\mathbb{E}}\left(X_{j}^n(i\delta)\right)$
then $Z_j^n(t)\sim\mathcal N(0,\omega_j^n(t))$, and then by making use of
the central moments of Gaussian random variables, we have
\begin{align}
&\tilde{\mathbb E}\left[\left(\sum_{i=1}^{[t/\delta]}\sum_{j=1}^n\sqrt n\xi_{i\delta}^n\bar
a_{j}^n(i\delta-)\left[\varphi(X_{j}^n(i\delta))-\tilde{\mathbb{E}}\left(\varphi(X_{j}^n(i\delta))\right)\right]\right)^{12}\Bigg|\mathcal
Y_{i\delta-}\right]\nonumber\\
\leq&2\|\varphi'\|_{0,\infty}^{12}\tilde{\mathbb E}\left[\left(\sum_{i=1}^{[t/\delta]}\sum_{j=1}^n\sqrt
n\xi_{i\delta}^n\bar
a_{j}^n(i\delta-)Z_j^n(i\delta)\right)^{12}\Bigg|\mathcal
Y_{i\delta-}\right]\nonumber\\
\leq&C^T\|\varphi\|_{1,\infty}^{12}\|\sigma\|_{0,\infty}^{12}\delta^6n^9\sum_{j=1}^n\left(\xi_{i\delta}^n\bar
a_{j}^n(i\delta-)\right)^{12};\nonumber
\end{align}
then by taking the expectation on both sides, we have
\begin{align}
\tilde{\mathbb E}\left[\left(\sqrt nG_{[t/\delta]}^{n,\varphi}\right)^{12}\right]\leq&C^T\|\varphi\|_{1,\infty}^{12}\|\sigma\|_{0,\infty}^{12}\delta^6n^9\sum_{j=1}^n\tilde{\mathbb
E}\left[\left(\xi_{i\delta}^n\bar a_{j}^n(i\delta-)\right)^{12}\right]\nonumber\\
\leq&C^T\|\varphi\|_{1,\infty}^{12}\|\sigma\|_{0,\infty}^{12}\delta^6n^9\sum_{j=1}^n\sqrt{\tilde{\mathbb
E}\left[(\xi_{i\delta}^n)^{24}\right]\tilde{\mathbb E}\left[\left(\bar a_j^n(i\delta-)\right)^{24}\right]}
\leq\frac{\beta_{\varphi,\sigma,\delta}^T}{n^2},\nonumber
\end{align}
where
$$
\beta_{\varphi,\sigma,\delta}^T=C^T\sqrt{c_1^{T,24}e^{c_{24}T}}\|\varphi\|_{1,\infty}^{12}\|\sigma\|_{0,\infty}^{12}\delta^6
$$
is a constant independent of $n$. Then similar to the proof of Lemma \ref{lem.asymptotics_of_omega}, we have the result.
\end{proof}

\begin{lemma}[\eqref{eq.observation_three}]
Assume the conditions in Proposition \ref{prop.verify_condition2} hold, then
\begin{align}
\lim_{n\rightarrow\infty}\frac{1}{\sqrt n}\sum_{j=1}^n\int_{0}^{t}\xi_{[s/\delta]\delta}^na_j^n(s)R_{s,j}^1(\varphi)ds=\Lambda_t^{R^1,\varphi},
\end{align}
where
\begin{equation}\label{eq.Lambda_R0R1}
\Lambda_t^{R^1,\varphi}=c_\omega\int_0^t\tilde\rho_s^1(\Psi\varphi)ds;
\end{equation}
$c_\omega$ is a constant, and the operator $\Psi$ is defined by
$$\Psi\varphi=\frac{f\varphi'''}{2}+\frac{\sigma\varphi^{(4)}}{4}-\frac{3(A\varphi)''}{2}.$$
\end{lemma}
\begin{proof}
Since
\begin{align}
&\lim_{n\rightarrow\infty}\frac{1}{\sqrt n}\sum_{j=1}^n\int_{0}^{t}\xi_{[s/\delta]\delta}^na_j^n(s)R_{s,j}^1(\varphi)ds\nonumber\\
=&\lim_{n\rightarrow\infty}\frac{1}{\sqrt n}\sum_{j=1}^n\int_{0}^{t}\xi_{[s/\delta]\delta}^na_j^n(s)\Bigg\{\omega_j^n(s)
\left[\left(\frac{f\varphi'''}{2}+\frac{\sigma\varphi^{(4)}}{4}\right)(v_j^n(s))-I_j(A\varphi)\right]\Bigg\}ds\nonumber\\
=&\lim_{n\rightarrow\infty}c_\omega\int_{0}^t\tilde\rho_s^{n,1}(\Psi\varphi)ds=c_\omega\int_0^t\tilde\rho_s^1(\Psi\varphi)ds,\nonumber
\end{align}
we have the required result.
\end{proof}

\begin{lemma}[\eqref{eq.observation_four}]
Assume the conditions in Proposition \ref{prop.verify_condition2} hold, then
\begin{align}
\lim_{n\rightarrow\infty}\left\langle\frac{1}{\sqrt n}\sum_{j=1}^n\int_{0}^{\cdot}\xi_{[s/\delta]\delta}^na_j^n(s)R_{s,j}^2(\varphi)dY_s\right\rangle_t
=\left\langle\Lambda_\cdot^{R^2,\varphi}\right\rangle_t,
\end{align}
where
\begin{equation}
c_\omega\int_{0}^t\left(\tilde\rho_s^1(h\varphi''-(h\varphi)'')\right)dB_s^{(2)},
\end{equation}
$c_\omega$ is a constant and $B^{(2)}$ is a Brownian motion independent of
$Y$.
\end{lemma}
\begin{proof}
Observe that
\begin{align}
&\lim_{n\rightarrow\infty}\left\langle\int_{0}^{\cdot}\frac{1}{\sqrt n}\sum_{j=1}^n\xi_{[s/\delta]\delta}^na_j^n(s)R_{s,j}^2(\varphi)dY_s\right\rangle_t\nonumber\\
=&\lim_{n\rightarrow\infty}\int_{0}^{t}\left(\frac{1}{2\sqrt n}\sum_{j=1}^n\xi_{[s/\delta]\delta}^na_j^n(s)\omega_j^n(s)\left[(h\varphi''-(h\varphi)'')(v_j^n(s))\right]\right)^2ds\nonumber\\
=&\lim_{n\rightarrow\infty}c_\omega^2\int_{0}^{t}\left(\tilde\rho_s^{n,1}\left(h\varphi''-(h\varphi)''\right)\right)^2ds
=c_\omega^2\int_{0}^{t}\left(\tilde\rho_s^{1}\left(h\varphi''-(h\varphi)''\right)\right)^2ds=\left\langle\Lambda_\cdot^{R^2,\varphi}\right\rangle_t;\nonumber
\end{align}
and then we have the result.
\end{proof}

\begin{lemma}
Assume the conditions in Proposition \ref{prop.verify_condition2} hold, then
\begin{align}
\lim_{n\rightarrow\infty}\left\langle\frac{1}{\sqrt n}\sum_{j=1}^n\int_{0}^{\cdot}\xi_{[s/\delta]\delta}^na_j^n(s)R_{s,j}^3(\varphi)dV_s^{(j)}\right\rangle_t=\left\langle\Lambda_\cdot^{R^3,\varphi}\right\rangle_t,
\end{align}
where 
$$
\Lambda_t^{R^3,\varphi}=\int_0^t\sqrt{\tilde\rho_s^2\left((\sigma\varphi')^2\right)}dB_s^{(3)},
$$
$B^{(3)}$ is a Brownian motion independent of $B^{(2)}$ and $Y$.
\end{lemma}
\begin{proof}
Note that $\omega_j^n\propto\frac{1}{\sqrt n}$, then by
the same
approach as in the proof of Lemma \ref{lem.asymptotics_of_omega}
in \cite{Li}, we have
\begin{align}
&\lim_{n\rightarrow\infty}\left\langle\frac{1}{\sqrt n}\sum_{j=1}^n\int_{0}^{\cdot}\xi_{[s/\delta]\delta}^na_j^n(s)R_{s,j}^3(\varphi)dV_s^{(j)}\right\rangle_t\nonumber\\
=&\lim_{n\rightarrow\infty}\int_0^t\tilde\rho_s^{n,2}\left((\sigma\varphi')^2\right)ds=\int_0^t\tilde\rho_s^2\left((\sigma\varphi')^2\right)ds=\left\langle\Lambda_\cdot^{R^3,\varphi}\right\rangle_t.
\end{align}
We then have the result.
\end{proof}




\addcontentsline{toc}{chapter}{Bibliography}
\begin{thebibliography}{99}


\bibitem{Bain and Crisan} A. Bain and D. Crisan, \emph{Fundamentals of Stochastic
Filtering}, Stochastic Modelling and Applied Probability, vol. 60, Springer,
2008.


\bibitem{Chopin} N. Chopin, ``Central limit theorem for sequential Monte Carlo methods and its application to Bayesian inference,'' \emph{Ann. Statist.,} vol. 32(6), pp. 2385-2411, 2004.

\bibitem{Crisan 4} D. Crisan, ``Superprocesses in a Brownian environment,''
\emph{Proc. R. Soc. Lond. Ser. A. Math. Phys. Eng. Sci.,} vol. 460(2041),
pp. 243-270, 2004.

\bibitem{Crisan and Li} D. Crisan and K. Li, ``Generalised particle filters with Gaussian measures,'' \emph{Proceedings of 19th European Signal Processing Conference}, pp. 659-663, 2011.

\bibitem{Crisan and Li 2} D. Crisan and K. Li, ``Generalised particle filters with Gaussian mixtures,'' \emph{arXiv:1306.0255}, 2013.


\bibitem{Obanubi} D. Crisan and O. Obanubi, ``Particle filters with random
resampling times,''
\emph{Stochastic Processes and their Applications}, vol. 122, pp. 1332-1368,
Jan. 2012.

\bibitem{Crisan and Xiong} D. Crisan and J. Xiong, ``A central limit type theorem for a class of particle filters,'' \emph{Comm. Stoch. Anal.,} vol. 1, pp. 103-122, 2007.

\bibitem{Ondrejat} M. Ondrej\'at, ``Uniqueness for stochastic evolution equations
in Banach spaces,'' \emph{Dissertationes Math. (Rozprawy Mat.)}, 2004.


\bibitem{crisroz}
D.  Crisan, and B. Rozovsky, editors,
\emph{The Oxford Handbook of Nonlinear Filtering},
Oxford University Press, 2011.

\bibitem{Del Moral CLT}
P. Del Moral, \textit{Feynman-Kac Formulae: Genealogical and
Interacting Particle Systems with Applications}. New York: Springer, 2004.

\bibitem{Del Moral and Guionnet} P. Del Moral and A. Guionnet, ``Central limit theorem for nonlinear filtering and interacting particle system,'' \emph{Ann. Appl. Probab.,} vol. 9(2), pp. 275-297, 1999.

\bibitem{Del Moral and Miclo} P. Del Moral and L. Miclo, \emph{Branching and interacting particle systems approximations of Feynman-Kac formulae with applications to non-linear filter.} S\'eminaire de Probabilit\'es, XXXIV, 1-145, Lecture Notes in Math, 1729, Springer, Berlin, 2000. 

\bibitem{Ethier and Kurtz} S. N. Ethier and T. G. Kurtz, \emph{Markov Processes,}
Wiley Series in Probability and Mathematical Statistics: Probability and
Mathematical Statistics, John Wiley \& Sons Inc., New York, 1986. Characterization
and convergence.

\bibitem{Kallianpur and Karandikar} G. Kallianpur and R. L. Karandikar, ``White noise calculus and nonlinear filtering theory,'' \emph{Annals of Probability}, vol. 13(4), pp. 1033-1107, 1985.

\bibitem{Karatzas and Shreve} I. Karatzas and S.E. Shreve, \emph{Brownian
Motion and Stochastic Calculus}, second ed., Graduate Texts in Mathematics,
vol. 113, Springer-Verleg, N.Y., 1991.

\bibitem{Kunsch} H.R. Kunsch, ``Recursive Monte Carlo filters: algorithms and theoretical analysis,'' \emph{Ann. Statist.,} vol. 33(5), pp. 1983-2021, 2005.

\bibitem{Kurtz and Protter} T.G. Kurtz and P. Protter, ``Weak limit theorems for stochastic integrals and stochastic differential equations,'' \emph{Annals of Probability}, vol. 19(3), pp. 1035-1070, 1991.

\bibitem{Kushner} H.J. Kushner, ``Approximations to optimal nonlinear flters,''
\emph{IEEE Trans. Automatic Control,} vol. 12(5), pp. 546-556, Oct. 1967.

\bibitem{Li} K. Li, ``Generalised particle filters,'' \emph{PhD Thesis, Imperial College London}, UK, 2013.

\bibitem{Lucic and Heunis} V.M. Lucic and A.J. Heunis, ``On uniqueness of
solutions for the stochastic differential equations of nonlinear filtering,''
\emph{Annals of Applied Probability}, vol. 11(1), pp. 182-209, 2001.

\bibitem{Rockner et al} M. R\"ockner, B. Schmuland, and X. Zhang, ``Yamada-Watanabe theorem for stochastic evolution equations in infinite dimensions,'' \emph{Condensed Matter Physics}, vol. 2(54), pp. 247-259, 2008.

\bibitem{Roelly-Coppoletta} S. Roelly-Coppoletta, ``A criterion of convergence
of measure-valued processes: application to measure branching processes,''
\emph{Stochastics}, vol. 17(1-2), pp. 43-65, 1986.  

\bibitem{Rogers and Williams} C. Rogers and D. Williams, \emph{Diffusions,
Markov Processes and Martingales: Volume I Foundations}, second ed., Cambridge,
UK, Cambridge University Press, 2000.

\bibitem{Stratonovich 1} R. L. Stratonovich, ``On the theory of optimal non-linear filtering of random functions,'' \emph{ Theory of Probability and its Applications,} vol. 4, pp. 223-–225, 1959.

\bibitem{Xiong and Zeng} J. Xiong and Y. Zeng, ``A branching particle approximation to a filtering micromovement model of asset price,'' \emph{Stat. Inference Stoch. Processes,} vol. 14, pp. 111-140, 2011.

\bibitem{Zakai} M. Zakai, ``On the optimal filtering of diffusion processes,'' \emph{Z. Wahrscheinlichkeitstheorie und Verw}, Gebiete, vol. 11, pp. 230-243, 1969.

\end{thebibliography}
\end{document}